\documentclass[a4paper,12pt]{article}
\usepackage{amsmath, amsthm,amssymb, latexsym}
\usepackage{tikz,hyperref,pgf,pgflibraryarrows,enumerate, amscd, amssymb}
\usepackage[all]{xy}
\usepackage[utf8]{inputenc}
\usepackage[T1]{fontenc}
\usepackage{url}
\usepackage{mathbbol}
\usepackage{comment}
\usepackage{paralist}

\newtheorem{thm}{Theorem}[section]

\newtheorem{lem}[thm]{Lemma}
\newtheorem{prop}[thm]{Proposition}

\newtheorem*{pr}{Proposition}

\theoremstyle{definition}

\theoremstyle{remark}

\newtheorem{example}[thm]{Example}

\numberwithin{equation}{section}

\newcommand{\Z}{\mathbb{Z}}

\newcommand{\ZZ}{\Z}
\newcommand{\RR}{\mathbb{R}}
\newcommand{\TT}{\mathbb{T}}

\newcommand{\CC}{\mathbb{C}}

\newcommand{\DD}{D^{c}_{\mu\nu}}

\newcommand{\Id}{\operatorname{Id}}

\newcommand{\EE}{E^c_{\mu\nu}}

\newcommand{\YM}{\operatorname{YM}}
\newcommand{\Tr}{\operatorname{Tr}}

\newcommand{\id}{\operatorname{id}}

\usepackage[margin=0.9in]{geometry}
\AtEndDocument{\bigskip{\small%
  \textsc{Sooran Kang: College of General Education, Chung-Ang University, 84 Heukseok-ro, Dongjak-gu, Seoul, Republic of Korea} \par
  \textit{E-mail address}: \texttt{sooran09@cau.ac.kr} \par
  \addvspace{\medskipamount}
  \textsc{Franz Luef: Department of Mathematical Sciences, Norwegian University of Science and Technology, Trondheim, Norway} \par
  \textit{E-mail address}, \texttt{franz.luef@ntnu.no} \par
  \addvspace{\medskipamount}
  \textsc{Stine Marie Berge: Institute of Analysis, Leibniz Universit\"at Hannover, Welfengarten 1, 30167 Hannover, Germany} \par
  \textit{E-mail address}, \texttt{stine.berge@math.uni-hannover.de}
}}

\title{Yang-Mills connections and sigma-models on quantum Heisenberg manifolds}

\author{Stine Marie Berge, Sooran Kang and Franz Luef}

\begin{document}

\maketitle

\begin{abstract}
We construct a spectral triple on a quantum Heisenberg manifold, which generalizes the results of Chakraborty and Shinha, and associate to it an energy functional on the set of projections, following the approach of Mathai-Rosenberg to non-linear sigma models. The spectral triples that we construct extend the  We derive a lower bound for this energy functional that is linked on the topological charge of the projection which depends on the curvature of a compatible connection. A detailed study of this lower bound is given for the Kang projection in quantum Heisenberg manifolds. These results display an intriguing interplay between non-linear sigma models and Yang-Mills theory on quantum Heisenberg manifolds, unlike in the well-studied case of noncommutative tori.
\end{abstract}
\section{Introduction}
The present paper points out an intriguing relation between Yang-Mills theory and non-linear sigma models for quantum Heisenberg manifolds. 
Yang-Mills theory is a cornerstone of modern physics \cite{YaMi} and has great importance in mathematics as well. Let us just mention the groundbreaking work of Donaldson \cite{Don83,Don87} on the geometry of four-dimensional manifolds and Yang-Mills theory. Extensions of the Yang-Mills theory  to the realm of noncommutative manifolds have been developed by Connes and Rieffel in the seminal work \cite{CR}, and by Dabrowski, Krajewski and Landi in \cite{DKL00,DKL03}. 
Non-linear sigma models originated in a quantum field theory. In short, non-linear sigma models describes a scalar field that takes on values in a Riemann manifold called the target manifold. These objects are of great relevance for several areas of physics, see \cite{Lind}. 

Yang-Mills theory and sigma models for noncommutative manifolds have been extensively studied for noncommutative tori \cite{CR,DKL00}. This is mainly due to the control over the vector bundles over noncommutative tori, also called Heisenberg modules, developed in the seminal work of Connes \cite{co80} and Rieffel's contribution in \cite{ri88}. Quantum Heisenberg manifolds (QHMs), introduced by Rieffel in \cite{Rie3}, is another class of $C^*$-algebras denoted by $\DD$. For $\DD$ we have  access to a class of projective modules $\Xi$ due to the substantial contributions of Abadie \cite{Ab2,Ab3,Ab5,AE}.

During the last decade numerous researchers have developed a Yang-Mills theory for projective modules over QHMs \cite{Ch1,ChG,ChS,Kang1,Lee}. This is for two reasons more involved than the case of noncommutative tori: 
\begin{inparaenum}[(i)]
\item QHMs are not finitely generated, and 
\item the existence of connections of non-constant curvature on vector bundles over a QHM.
\end{inparaenum}
In \cite{KLP}, two of the authors constructed a number of Yang-Mills connections for QHMs, demonstrating the intricate nature of Yang-Mills theory in this case.  

Sigma models for noncommutative manifolds are based on energy functionals on the set of projections of the $C^*$-algebra \cite{DKL00,MaRos}. A striking feature of the sigma model for noncommutative tori is a lower bound of the energy functional in terms of the first Connes-Chern number of a projection \cite{DKL00}. In the classical context, this was first noted in \cite{BePo75}. Recent progress on sigma models for noncommutative tori has been due to the interpretation of generators of Heisenberg modules as Gabor frames \cite{AE,Luef}. Hence, one can apply results from Gabor frames, such as the duality principle and examples of Gabor frames, to construct new solitons over noncommutative tori \cite{DLL,DJLL}. The relevance of Gabor frames in this context has also been used by Lee in his work on sigma model solitons \cite{Lee18,Lee20}. 

The non-linear sigma model for noncommutative tori in \cite{DKL00} is based on an energy functional for the set of projections. This generalizes the Dirichlet energy functional in the setting of Riemannian manifolds. Mathai and Rosenberg have constructed sigma models for noncommutative manifolds based on a spectral triple \cite{MaRos}. This construction for the standard spectral triple for noncommutative tori produces the same energy functional as in \cite{DKL00}.
Recall that a spectral triple, introduced by Connes, is a natural noncommutative generalization of a compact manifold with certain summability properties. In the case of compact manifolds, the summability properties correspond to the dimension of the manifold, see \cite{C2}. Spectral triples for $C^*$-algebras have greatly advanced the theory of noncommutative geometry due to its numerous applications in mathematics and physics, e.g.\ index theory. 

We aim to study sigma models for quantum Heisenberg manifolds. The energy functional for the set of projections in QHMs is based on the construction in \cite{MaRos} for the spectral triple on QHMs proposed in \cite{ChS}. Since we have not been able to follow the arguments in \cite{ChS} concerning the spectral triple, we have included our own approach to the construction of spectral triples on QHMs. 
Consequently, we obtain an energy functional
\begin{equation}\label{eq:energy_functional}
S(p)=2 \,\tau_D\big((\delta_X p)^2 + (\delta_Y p)^2+ (\delta_Z p)^2\big),
\end{equation}
on the set of projections in QHMs, see Section \ref{sec:prelim} for undefined notions.

We find, as in the case of noncommutative tori, that the energy functional \eqref{eq:energy_functional} has a lower bound in terms of a topological charge of a projection. This lower bound has a very different flavor than the one obtained for Heisenberg modules over noncommutative tori in \cite{DKL00,DLL,DJLL} since we have to handle connections with non-constant curvature. Consequently, this leads to a lower bound depending on the connection. Concretely, we have the following result. 
\begin{pr}
Let $R\in \Xi$ be the element that gives the Grassmannian connection in \eqref{eq:Gr-conn} such that $\langle R, R \rangle_R^D=Q$ and $\langle R, R \rangle_L^E=\Id_E$. Let $\nabla$ be a compatible connection with curvature $\Theta_\nabla$.
Then the topological charge of the projection $Q$ is given by
\[
c_\nabla(Q)=-\frac{1}{2\pi i} \tau_E(\langle R, (\Theta_\nabla(X,Y) +\nabla_{[X,Y]} +\Theta_\nabla(X,Z)+\Theta_\nabla(Y,Z))\cdot R\rangle_L^E).
\]
\end{pr}

This lower bound has quite intriguing features based on the Yang-Mills theory of QHMs, which in a novel way links sigma models and Yang-Mills theories for QHMs. We compute this lower bound for some of the Yang-Mills connections discussed in \cite{KLP} using the Kang projection in QHMs \cite{Kang1} which is compactly supported. Note that this connection between Yang-Mills theory and sigma models does not appear for noncommutative tori, since in this case all Yang-Mills connections have constant curvature. 
\\~~\\
{\bf Acknowledgment:} Sooran Kang was supported by the National Research Foundation of Korea (NRF) grant funded by the Korea government (MIST) (No.~2017R1D1A1B03034697, No.~2020R1F1A1A01076072). 

\section{Preliminaries}\label{sec:prelim}

\subsection{Quantum Heisenberg manifolds} \label{subsec:prel-QHMs}
In this section, we briefly review the finitely projective module $\Xi$ over the quantum Heisenberg manifold $\DD$ constructed by Abadie \cite{Ab1}. We also discuss several compatible linear connections on $\Xi$ given in \cite{Kang1, Lee, KLP}.

Let $M=\RR\times\TT$ and denote by $C_b(M)$ the set of continuous bounded functions on $M$. Consider the commuting actions of $\mathbb{Z}$ on $M$ denoted by $\lambda$ and $\sigma$ defined by
\begin{equation*}
\lambda_p(x,y)=(x+2 p\mu,y+2 p\nu)\quad\text{and}\quad
\sigma_p(x,y)=(x-p,y),
\end{equation*}
where  $\mu, \nu \in \mathbb{R}$ and $p\in \mathbb{Z}$.
Then construct the crossed product $C^{\ast}$-algebras $C_b(M)\rtimes_{\lambda}\mathbb{Z}$ and $C_b(M)\rtimes_{\sigma}\mathbb{Z}$ with their usual star-product and involution. In the case of the algebra $C_b(M)\rtimes_{\lambda}\mathbb{Z}$ the star-product and involution become
\[\begin{split}
(\Phi \ast \Psi)(x,y,p)=\sum_{q\in \ZZ} \Phi(x,y,q) \Psi(x-2q \mu,y- 2 q \nu, p-q),\\
\Phi^\ast(x,y,p)=\overline{\Phi}(x-2 p \mu, y-2 p \nu, -p),
\end{split}\]
for $\Phi, \Psi\in C_b(M)\rtimes_{\lambda}\mathbb{Z}$.

On the algebras $C_{b}(M)\rtimes_{\lambda}\mathbb{Z}$ and $C_{b}(M)\rtimes_{\sigma}\mathbb{Z}$ we let $\rho$ and $\gamma$ denote the actions of $\mathbb{Z}$ defined on \[\Phi \in C_c(M\times \mathbb{Z})=\{\Phi\in C_b(M\times\ZZ)\colon \Phi\text{ compact in }\ZZ\}\] 
by
\begin{equation}\label{eq:fixed-action}
(\rho_k\Phi)(x,y,p)=\overline{e}(ckp(y- p\nu))\Phi(x+k,y,p),
\end{equation}
\begin{equation*}
(\gamma_k\Phi)(x,y,p)=e(cpk(y- k\nu))\Phi(x-2 k\mu,y-2 k\nu, p),
\end{equation*}
 where $k, p \in \mathbb{Z}$ and $e(x)=\exp(2\pi ix)$ for $x\in \RR$.
The actions $\rho$ and $\gamma$ can be shown to be proper.
Denote by $D_0$ the $\ast$-subalgebra in the multiplier algebra of $C_{b}(M)\rtimes_{\lambda}\mathbb{Z}$ consisting of functions $\Phi \in C_c(M\times \mathbb{Z})$ that satisfies $\rho_k(\Phi)=\Phi$ for all $k \in \mathbb{Z}$. The generalized fixed point algebra of $C_{b}(M)\rtimes_{\lambda}\mathbb{Z}$ by the action $\rho$ is denoted by $\DD$ and is the closure of the $\ast$-subalgebra $D_0$. 
Similarly, the generalized fixed point algebra of $C_{b}(M)\rtimes_{\sigma}\mathbb{Z}$ by the action $\gamma$ is denoted by $E^{c}_{\mu\nu}$. The algebra $E^{c}_{\mu\nu}$ is the closure of the $\ast$-subalgebra $E_0$ in the multiplier algebra of $C_{b}(M)\rtimes_{\sigma}\mathbb{Z}$ consisting of functions $\Psi \in C_c(M\times \mathbb{Z})$ with compact support on $\mathbb{Z}$ and that satisfies $\gamma_k(\Psi)=\Psi$ for all $k \in \mathbb{Z}$.
 
According to the general theory of representations in \cite[Sec.~7.7]{Ped} the induced representation $\pi$ of $\DD$ on $\mathcal{H}=L^2(M\times \ZZ)$ is given by
\begin{equation}\label{eq:hilbert_space_rep}
(\pi(\Phi) \xi) (x,y,p)= \sum_{q\in \ZZ} \Phi(x-2 p \mu, y-2 p \nu, q) \xi(x,y,p-q)
\end{equation}
for $\xi\in L^2(M\times \ZZ)$ and $\Phi\in C_c^\infty (M\times \ZZ)$. Note that Rieffel's representation in [22] differs from the one used here. 

According to the main theorem in \cite{Ab1}, these generalized fixed point algebras $D^{c}_{\mu\nu}$ and $E^{c}_{\mu\nu}$ are strongly Morita equivalent. Let $\Xi$ be the left-$E^{c}_{\mu\nu}$ and right-$D^{c}_{\mu\nu}$ bimodule constructed as follows: The bimodule $\Xi$ is the completion of $C_c(M)$ with respect to either one of the norms induced by one of the $D^{c}_{\mu\nu}$- and $E^{c}_{\mu\nu}$-valued inner products, $\langle\cdot,\cdot\rangle_R^D$ and $\langle\cdot,\cdot\rangle_L^E$ respectively. The inner products are given by
\begin{equation}\label{D-value-inner}
\langle f,g \rangle_R^D(x,y,p)=\sum_{k \in \mathbb{Z}}\overline{e}(ckp(y- p\nu))f(x+k,y)\overline{g}(x-2 p\mu+k,y-2p\nu)\,,
\end{equation}
\begin{equation}\label{E-value-inner}
\langle f,g \rangle_L^E(x,y,p)=\sum_{k \in \mathbb{Z}}e(cpk(y- k\nu))\overline{f}(x-2 k\mu,y-2 k\nu)g(x-2 k\mu+p,
y-2 k\nu),
\end{equation}
 where $f,g \in C_c(M)$ and $p \in \mathbb{Z}$.
Also the left and right action of $E^{c}_{\mu\nu}$ and $D^{c}_{\mu\nu}$ on $\Xi$ are given by
\begin{equation*}\label{left-action}
(\Psi\cdot f)(x,y)=\sum_{q \in\mathbb{Z}}\overline{\Psi}(x,y,q)f(x+q,y),
\end{equation*}
\begin{equation*}\label{right-action}
(g\cdot\Phi)(x,y)=\sum_{q\in\mathbb{Z}}g(x+2 q\mu,y+2 q\nu)\overline{\Phi}(x+2 q\mu,y+2 q\nu,q),
\end{equation*}
for $\Psi \in \EE$ , $\Phi \in \DD$, and $f, g \in \Xi$. Moreover, for $f_1, f_2, f_3\in \Xi$ we have
\begin{equation}\label{eq:bimodule-asso}
\langle f_1, f_2 \rangle_L^E \cdot f_3= f_1 \cdot \langle f_2, f_3 \rangle_R^D.
\end{equation}
%Since $\hslash$ is the Planck constant, we let $\hslash=1$ from now on for notational convenience.

Let $H$ be the \textit{reparametrized Heisenberg group} given by Rieffel in \cite{Rie3} as follows: For $x,y,z\in \RR$ and a positive integer $c$, let
\[
(x,y,z):=\begin{pmatrix} 1 & y & z/c \\ 0 & 1 & x \\ 0 & 0 & 1 \end{pmatrix}.
\]
We can identify $H$ with $\RR^3$ equipped with the product
\[
(x,y,z)(x'y'z')=(x+x', y+y', z+z'+cyx').
\]
Then the action $L$ of $H$ on the quantum Heisenberg manifold $\DD$ is given by
 \begin{equation}\label{eq:action_L}
(L_{(r,s,t)}\Phi)(x,y,p)=e(p(t+cs(x-r-p\mu)))\Phi(x-r,y-s,p).
 \end{equation}
 We denote by $(\DD)^\infty \subset \DD$ the smooth subalgebra given by
 \[
 (\DD)^\infty=\{\Phi\in \DD: h\mapsto L_h(\Phi)\;\;\text{is smooth in norm for $h\in H$}\},
 \]
 The infinitesimal form of $L$ gives an action $\delta$ of the corresponding Heisenberg Lie algebra $\mathfrak{h}$ on $(\DD)^\infty$. In particular, we let $X, Y, Z$ be the basis of $\mathfrak{h}$ given by
 \begin{equation}\label{eq:Lie-alg}
 X=(0,1,0)=\begin{pmatrix} 0 & 1& 0 \\ 0 & 0 & 0 \\ 0 & 0 & 0 \end{pmatrix}, \;\; Y=(1,0,0)=\begin{pmatrix} 0 & 0 & 0 \\ 0 & 0 & 1 \\ 0 & 0 & 0 \end{pmatrix}, \;\; Z=(0,0,1)=\begin{pmatrix} 0 & 0 & 1/c \\ 0 & 0 & 0 \\ 0 & 0 & 0 \end{pmatrix}
 \end{equation}
 and then we have $[X,Y]=cZ$. The corresponding derivations on $(\DD)^\infty$ are given by
 \begin{equation}\begin{split}\label{eq:deri}
& \delta_X(\Phi)(x,y,p)=2\pi icp(x-p\mu)\Phi(x,y,p)-\frac{\partial \Phi}{\partial y}(x,y,p),\\
& \delta_Y(\Phi)(x,y,p)=-\frac{\partial \Phi}{\partial x}(x,y,p),\\
& \delta_Z(\Phi)(x,y,p)=2\pi ip\,\Phi(x,y,p),
\end{split}\end{equation}
for $\Phi\in (\DD)^\infty$. For an arbitrary $V=c_1X+c_2Y+c_3Z\in \mathfrak{h}$ we define the derivation $\delta_V$ by \[\delta_V=c_1\delta_X+c_2\delta_Y+c_3\delta_Z.\]

\subsection{Yang-Mills equations on the QHMs}\label{subsec:prel-YM}

According to Lemma~1 of \cite{C1}, for the left-$\EE$ and right-$\DD$ projective bimodule $\Xi$, there is a dense left-$(\EE)^\infty$ and right-$(\DD)^\infty$ submodule $\Xi^\infty$ of $\Xi$. For notational simplicity we omit $\infty$ from smooth spaces of $C^*$-algebras and projective modules over them from now on. 

Recall that $\mathfrak{h}$ denotes the Lie algebra of the Heisenberg group. We say that a linear map $\nabla:\Xi \to \Xi \otimes \mathfrak{h}$ is a \textit{linear connection} if it satisfies
\[
\nabla_V(\xi \cdot \Phi)=(\nabla_V(\xi))\cdot \Phi + \xi \cdot (\delta_V(\Phi)),
\]
for all $V\in \mathfrak{h}$, $\xi\in \Xi$ and $\Phi\in \DD$. We say that the linear connection $\nabla$ is \textit{compatible with respect to the Hermitian metric} $\langle \cdot, \cdot \rangle_R^D$ if 
\begin{equation}\label{eq:conne-comp-right}
\delta_V(\langle \xi,\eta\rangle_R^D)=\langle \nabla_V \xi, \eta\rangle_R^D + \langle \xi, \nabla_V \eta\rangle_R^D,
\end{equation}
for $\eta, \xi \in \Xi$.
Likewise, we say that $\nabla$ is \textit{compatible with respect to the Hermitian metric} $\langle \cdot, \cdot \rangle_L^E$ if 
\begin{equation}\label{eq:conne-comp-left}
\delta_V(\langle \xi,\eta\rangle_L^E)=\langle \nabla_V \xi, \eta\rangle_L^E + \langle \xi, \nabla_V \eta\rangle_L^E.
\end{equation}
If $\nabla $ is compatible with both Hermitian metrics, we will simply refer to $\nabla$ as a \textit{compatible connection}.
Then the \textit{curvature} of a compatible linear connection $\nabla$ is defined to be the alternating bilinear form $\Theta_\nabla$ on $\mathfrak{h}$, given by
\[
\Theta_\nabla(V,W)=\nabla_V \nabla_W - \nabla_W \nabla_V -\nabla_{[V,W]}
\]
for $V,W\in \mathfrak{h}$. From now on, we say ``connection'' when we mean ``linear connection''. 
We denote the set of compatible linear connections on $\Xi$ by $CC(\Xi)$. 

To define the Yang-Mills functional $\YM$ on $CC(\Xi)$, we need the notion of trace. Let $\tau_D$ be a faithful $L$-invariant trace on $\DD$, where $L$ is the action \eqref{eq:action_L}. Using $\tau_D$, we define the trace $\tau_E$ on $\EE$ by 
\begin{equation}
\tau_E(\langle \xi,\eta\rangle_L^E)=\tau_D(\langle \eta,\xi\rangle_R^D).
\end{equation}
According to \cite{Rie1}, there is a faithful trace on $\DD$ given by
\[
\tau_D(\Phi)=\int_{\TT^2} \Phi(x,y,0)\, dx\, dy
\]
for $\Phi\in \DD$, and hence the corresponding  $\tau_E$ can by computed as  
\begin{equation}\label{eq:trace_E}
\tau_E(\Psi)= \int_0^{2\mu}\int_0^1 \Psi(x,y,0)\, dx\, dy
\end{equation}
for $\Psi \in (\EE)_0$.

The \textit{Yang-Mills functional} $\YM$ is defined on $CC(\Xi)$ by
\[
\YM(\nabla)=-\tau_E(\{\Theta_\nabla,\Theta_\nabla\}_E),
\]
where $\{\cdot,\cdot\}_E$ is a bilinear form given by
\[
\{\alpha, \beta\}_E=\sum_{i<j}\alpha(Z_i,Z_j)\beta(Z_i,Z_j),
\]
for alternating $\EE$-valued 2-forms $\alpha, \beta$ and $\{Z_1,Z_2,Z_3\}$ is a basis for $\mathfrak{h}$. 
We are interested in the nature of the set of connections where $\YM$ attains its minimum and the set of critical points of $\YM$. In particular, we are interested in compatible connections which are both critical points and minimizers, called Yang-Mills connections. To be precise, we say that a compatible connection $\nabla$ is a \textit{global minimizer of $\YM$} if \[\YM(\nabla)\le \YM(\nabla')\]
for any other connection $\nabla'\in CC(\Xi)$. We say that $\nabla$ is a \textit{local minimizer of $\YM$ subject to the constant curvature constraint} if \[\YM(\nabla)\le \YM(\nabla')\] for any other connection $\nabla'$ with constant curvature.
According to \cite[Theorem~1.1]{Rie5-cp} and  \cite[Section 5]{Kang1}, a compatible connection $\nabla$ on $\Xi$ with curvature $\Theta_\nabla$ is a critical point of $\YM$ exactly when $\nabla$ satisfies the following equations:
\begin{equation}\label{eq:critical_point}
\begin{split}
&(1)\;[\nabla_Y, \Theta_\nabla(X,Y)]+[\nabla_Z,\Theta_\nabla(X,Z)]=0,\\
&(2)\;[\nabla_X,\Theta_\nabla(Y,X)]+[\nabla_Z,\Theta_\nabla(Y,Z)]=0,\\
&(3)\;[\nabla_X,\Theta_\nabla(Z,X)]+[\nabla_Y, \Theta_\nabla(Z,Y)]-c\Theta_\nabla(X,Y)=0.
\end{split}\end{equation}

There are several connections with specific formulas found in \cite{Kang1}, \cite{Lee} and \cite{KLP}. The rest of this section will be used to go through different explicit connections.
The most natural and standard connection is the \textit{Grassmannian connection}. For our case, the Grassmannian connection $\nabla^G$ on $\Xi$ is given by, for all $V\in \mathfrak{h}$
\begin{equation}\label{eq:Gr-conn}
\nabla^G_V(\xi)=R\cdot \delta_V(\langle R, \xi\rangle_R^D),
\end{equation}
where $R\in \Xi$ is chosen to satisfy $\langle R, R\rangle_R^D$ is a projection in $\DD$ and $\langle R, R\rangle_L^E=\id_E$, where $\Id_E(x,y,p)=\delta_0(p)$ is the multiplicative identity.

For later convenience, we describe the special choice function $R$ given in \cite{Kang1}. On the interval $(-2\mu, 0]$ define $R$ to be $0$ on $(-2 \mu, - \mu)$, smooth on $(- \mu, -\frac{1}{2} \mu)$ and $1$ on $[-\frac{1}{2} \mu, 0]$, where $| \mu|<\frac{1}{4}$. and define $R$ on $[0, 2 \mu)$ by 
\[R(x)=\sqrt{1-|R(x-2 \mu)|^2}.\] Elsewhere, the function $R$ is defined to be zero, making $R$ compactly supported and smooth. Moreover, because of the condition $\langle R,R\rangle_R^D$ is projection and $\langle R, R\rangle_L^E=\Id_E$, $R$ satisfies the equations (B-1), (B-2), and (B-3) in \cite{Kang1}. This implies that for $x\in [-2\mu,2\mu]$ we have that
\begin{equation}\label{eq:R-cond}
(R(x))^2+(R(x-2 \mu))^2=1, \quad \text{and}\quad R(x+2 q \mu)= 0 \;\;\text{except for $q=0$ and $q=\pm 1$}
\end{equation}
The curvature for this choice of $R$ is given by 
\[\begin{split}
&\Theta_{\nabla^G}(X,Y)(x,y,p)=f_1(x)\delta_0(p),\\
&\Theta_{\nabla^G}(X,Z)(x,y,p)=0,\\
&\Theta_{\nabla^G}(Y,Z)(x,y,p)=f_2(x)\delta_0(p),
\end{split}\]
where $f_1$ and $f_2$ are smooth skew-symmetric periodic functions. It can be shown that $\nabla^G$ is neither a critical point nor a minimizer of $\YM$.

There are a couple of sets of connections that give rise to constant curvature found in \cite{Lee} and \cite{KLP}.
The connection $\nabla^0$ found in \cite{Lee} is given by
\begin{equation}\label{eq:Lee_conn}
\begin{split}
&(\nabla^0_X\xi)(x,y)=-\frac{\partial \xi}{\partial y}(x,y)+\frac{\pi ci}{2\mu}x^2f(x,y),\\
&(\nabla^0_Y\xi)(x,y)=-\frac{\partial \xi}{\partial x}(x,y),\\
&(\nabla^0_Z\xi)(x,y)=\frac{\pi i x}{\mu}\xi(x,y).
\end{split}
\end{equation}
The corresponding curvature $\Theta_{\nabla^0}$ is given by
\begin{equation}\label{eq:min_curv}
\Theta_{\nabla^0}(X,Y)=0,\;\;\Theta_{\nabla^0}(X,Z)=0,\;\;\Theta_{\nabla^0}(Y,Z)=\frac{\pi i}{\mu}\Id_E.
\end{equation}
Note that the formulas of $\nabla^0$ and $\Theta_{\nabla^0}$ here look a bit different from the ones in \cite{Lee} because the settings are slightly different.  A full detailed proof that $\nabla^0$ is a compatible connection with the above curvature is given in Appendix A of \cite{KLP}. Note that the connection $\nabla^0$ is a critical point of $\YM$ and a (local) minimizer of $\YM$ subject to the constant constraint. Hence $\nabla^0$ is a Yang-Mills connection subject to the constant constraint, for details see Theorem~4.3 in \cite{KLP}. 

According to Theorem~4.8 of \cite{KLP}, there is another set of Yang-Mills connections on $\Xi$ given by $\nabla^1=\nabla^0+\mathbb{H}$, where $\mathbb{H}$ is the linear map from $\mathfrak{h}$ to the set of skew-symmetric elements of $\EE$ given by 
\begin{equation}\label{eq:skew-H}
\begin{split}
&\mathbb{H}_X (x,y,p)= i\, g_1(y) \delta_0(p)\\
&\mathbb{H}_Y(x,y,p)= i\, g_2(x) \delta_0(p)\\
&\mathbb{H}_Z(x,y,p)=0,
\end{split}
\end{equation}
where $(x,y,p)\in M\times \ZZ$, and $g_1$, $g_2$ are real-valued differentiable functions satisfying $g_1(y)=g_1(y-2p\nu)$, $g_2(x)=g_2(x-2p\mu)$. The corresponding curvature $\Theta_{\nabla^1}$ is given by
\begin{equation}\label{eq:curv-nabla-1}
\Theta_{\nabla^1}(X,Y)=0,\quad \Theta_{\nabla^1}(X,Z)=0,\quad \Theta_{\nabla^1}(Y,Z)=\frac{\pi i}{\mu}\Id_E.
\end{equation}
There are many functions $g_1$ and $g_2$ making $\nabla^1$ into a connection with constant curvature. For examples, see Example~4.9 of \cite{KLP}.

There is a connection $\nabla^2$ found in \cite{KLP} with constant curvature, that neither is a minimum of $\YM$ nor a critical point. The connection $\nabla^2$ is given by
\begin{equation}\label{eq:nabla-2}
\begin{split}
(\nabla^2_Xf)(x,y)&=-\frac{\partial f}{\partial y}(x,y)+\Big(\frac{\pi c i}{2\mu}x^2-\nu i x+\mu i y\Big)f(x,y)\\
(\nabla^2_Yf)(x,y)&=-\frac{\partial f}{\partial x}(x,y)\\
(\nabla^2_Zf)(x,y)&=\frac{\pi i x}{\mu}f(x,y)
\end{split}
\end{equation}
Then the corresponding curvature $\Theta_{\nabla^2}$ is given by
\begin{equation}\label{eq:curv-nabla-2}
\Theta_{\nabla^2}(X,Y)=\nu i \Id_E,\quad \Theta_{\nabla^2}(X,Z)=0,\quad \Theta_{\nabla^2}(Y,Z)=\frac{\pi i}{\mu}\Id_E.
\end{equation}

The authors of \cite{KLP} found connections with non-constant curvature in Theorem~5.1 of \cite{KLP} by adding certain skew-symmetric elements $\mathbb{G}$ to $\nabla^0$ given by \eqref{eq:Lee_conn}. Let $\nabla^3=\nabla^0+\mathbb{G}$, where
\begin{equation}\label{eq:G}\begin{split}
&\mathbb{G}_X (x,y,p)= 0\\
&\mathbb{G}_Y(x,y,p)= 0\\
&\mathbb{G}_Z(x,y,p)=i \cos\Big(\frac{\alpha \pi x}{\mu}\Big)  \, \delta_0(p),
\end{split}
\end{equation}
for $\alpha \in \RR \setminus \negthickspace \{0\}$. Then the connection $\nabla^3$ gives rise to non-constant curvature where $\Theta_{\nabla^3}(X,Y)\ne 0$. In particular, the corresponding curvature $\Theta_{\nabla^3}$ is given by
\begin{equation}\label{eq:conn-nabla3}\begin{split}
&\Theta_{\nabla^3}(X,Y)=-c\,\mathbb{G}_Z \ne 0, \\
&\Theta_{\nabla^3}(X,Z)=[\nabla^0_X, \mathbb{G}_Z]=0, \\
&\Theta_{\nabla^3}(Y,Z)=\frac{\pi i}{\mu}\Id_E+[\nabla^0_Y, \mathbb{G}_Z]=\frac{\pi i }{\mu} \Id_E+\frac{\partial \mathbb{G}_Z}{\partial x} \ne 0.
\end{split}\end{equation}
 This shows that $\Theta_{\nabla^3}$ cannot be constant because $\mathbb{G}_Z(x,y,p)=i \cos(\frac{\alpha \pi x}{\mu})  \, \delta_0(p)$ is not constant for $\alpha \in \RR \setminus \negthickspace \{0\}$.  Moreover, Theorem~5.1 of \cite{KLP} says that  there exists a triple $(c,\mu,\alpha)$ with $c\in \Z^+, \mu\in (0,1/2], \alpha\in \RR\setminus \negthickspace \{0\}$ such that $\nabla^3$ is not a critical point of $\YM$. Additionally, for this choice of $c$, $\mu$ and $\alpha$ we have that $\nabla^3$ satisfies \[\YM(\nabla^3)  < \YM(\nabla^0)=\frac{2\pi^2}{\mu}.\]

\section{Spectral triples on the QHMs}
The goal of this section is to describe a class of spectral triples on QHMs. A large part of the section is devoted to showing that the objects defined are indeed spectral triples. These spectral triples are similar to the spectral triples given in \cite{ChS}, and the proofs in this section are inspired by their approach.

To begin, recall that an \textit{odd spectral triple} $(\mathcal{A},\mathcal{H},D)$ consists of a Hilbert space $\mathcal{H}$, an involutive unital algebra $\mathcal{A}$ of operators on $\mathcal{H}$, and a densely defined self-adjoint operator $D$ on $\mathcal{H}$ satisfying:
\begin{enumerate}[i)]
    \item $\|[a,D]\|<\infty$ for any $a\in \mathcal{A}$,
    \item $(D-\lambda)^{-1}$ is a compact operator for any $\lambda\in \CC\setminus\RR$.
\end{enumerate}
The operator $D$ will be referred to as the \textit{Dirac operator}. For a compact operator $K$ we denote by $\mu_m$ the $m$'th non-zero eigenvalue in decreasing order. Define the operator $|D|^{-n}$ for $n>0$ to be zero on $\ker(D)$ and the inverse of $|D|^{n}$ on the range of $D$. A odd spectral triple is said to be of \textit{dimension} $n>0$ if $|D|^{-n}$ is a compact operator with eigenvalues $\mu_m$ satisfying for a constant $C$ the inequality $\mu_m\le Cm^{-1}$. Equivalently, one can ask for the \textit{Dixmier trace} of $|D|^{-n}$ to be finite. Recall that for the Dixmier trace of a compact operator $K$ to be defined we need that
\[\lim_{N\to \infty}\frac{1}{\ln(N)}\sum_{m=0}^{N}\mu_m<\infty.\]
We will denote the ideal of all compact Dixmier trace class operators by $\mathcal{L}^{(1,\infty)}$.

Recall that the basis elements $X,Y, Z$ of $\mathfrak{h}$ are given in \eqref{eq:Lie-alg} and the derivations $\delta_X$, $\delta_Y$, and $\delta_Z$ are given in \eqref{eq:deri}. We denote the \textit{Pauli spin matrices} by 
\begin{equation}\label{eq:Pauli_matrix}
\sigma_X=\begin{pmatrix} 0 & 1 \\ 1 & 0 \end{pmatrix},\quad  \sigma_Y=\begin{pmatrix} 0 & -i \\ i & 0 \end{pmatrix},\quad\text{and}\quad\sigma_Z=\begin{pmatrix} 1 & 0 \\ 0 & -1 \end{pmatrix}.
\end{equation}

Let $\mathcal{A}\subset D_0$ be the set of all smooth functions in $D_0$. Note that $\mathcal{A}$ is a subalgebra of $\DD$. Additionally, one can view the condition $\rho(\Phi)=\Phi$ as a ``periodicity'' condition on $\mathcal{A}$. Define the space
\[V_b=\{\phi\in C_b([0,1]\times \mathbb{T}\times \mathbb{Z}) \colon \phi(0,y,p)=\Bar{e}(cp(y-p\nu))\phi(1,y,p)\}.\]
Then we have that 
$\left.\mathcal{A}\right|_{[0,1]\times \TT\times \ZZ}\subset V_b$. Additionally, for any $\phi \in \left.\DD\right|_{[0,1]\times \TT\times \ZZ}$ we can regain the function on $\mathcal{A}$ by the map
\[\tilde{\phi}(x,y,p)=e(c\lfloor x\rfloor(y-p\nu)) \phi(x-\lfloor x\rfloor, y,p).\]
Using this identification the derivations $\DD$ we define 
\[\tilde{\delta}_X\phi=\left.(\delta_X\tilde{\phi})\right|_{[0,1]\times\mathbb{T}\times\mathbb{Z}},
\quad \tilde{\delta}_Y\phi=\left.(\delta_Y\tilde{\phi})\right|_{[0,1]\times\mathbb{T}\times\mathbb{Z}},\quad \text{and } \tilde{\delta}_Y\phi=\left.(\delta_Y\tilde{\phi})\right|_{[0,1]\times\mathbb{T}\times\mathbb{Z}}.\]
To avoid notational clutter, we will often write $\delta$ instead of $\tilde{\delta}$.
By abuse of notation, we will use the same notation for $\mathcal{A}$ and $\left.\mathcal{A}\right|_{[0,1]\times\TT\times\ZZ}$.
Denote by $\pi$ the representation on $L^2(M\times \ZZ)$ given in \eqref{eq:hilbert_space_rep}. Using this representation we can define a representation $\tilde{\pi}$ of $\mathcal{A}$ acting on $\mathcal{H}=L^2([0,1]\times \TT\times \ZZ)$ by 
\[
(\tilde{\pi}(a) \xi) (x,y,p)= \sum_{q\in \ZZ} \tilde{a}(x-2p \mu, y-2 p \nu, q) \xi(x,y,p-q),%=(\tilde{\xi}*(\tilde{\Phi}^*))(x,y,p).
\]
for $a\in \mathcal{A}$ and $\xi \in \mathcal{H}$. Hence using $\tilde{\pi}$ we can view $\mathcal{A}$ as an algebra of operators on $\mathcal{H}$. Notice that the unit $a(x,y,p)=\delta_0(p)$ is contained in $\mathcal{A}$.

In the sequel we denote by $H^1(S)$ on an open set $S$ of $\mathbb{R}^2$ the Sobolev space of all elements in $L^2(S)$ possessing one weak derivative that lies in $L^2(S)$. If $S$ is a Lipschitz domain, then one can define boundary values of functions in $H^1(S)$ by using the trace theorem.

\begin{thm}[Spectral triple]\label{thm:spectral tripple}
Let $\mathcal{A}$ be as above and denote by $\mathcal{H}=L^2([0,1]\times \TT\times\ZZ)$.  Define an operator $D$ on $\mathcal{H}\otimes {\mathbb{C}}^2$ given by
\[
D(\Phi\otimes u)=\sum_{j=X,Y,Z}\tilde{\delta}_j(\Phi)\otimes \sigma_j(u),
\]
where $\sigma_j$'s are the Pauli spin matrices. The domain of the operator $D$ consists of all those square integrable functions $\Phi$ defined on $\mathcal{H}$ satisfying the conditions 
\begin{enumerate}[i)]
    \item $\Phi(\cdot,\cdot,p)\in H^1([0,1]\times \mathbb{T})$ for all $p\in \ZZ$,
    \item $\Phi(0,y,p)=\overline{e}(cp(y-p\nu)) \Phi(1, y, p)$ (interpreted by the use of the trace theorem),
    \item $p\Phi$ and $\partial_y(e(\lfloor cp^2\mu\rfloor y)\Phi)$ are in $\mathcal{H}$.
\end{enumerate}
For any $a\in \mathcal{A}$ we can define the operator $\tilde{\pi}(a)$ on $\mathcal{H}\otimes \CC^2$ by 
\[\tilde{\pi}(a)(\xi\otimes u)=(\tilde{\pi}(a)\xi)\otimes u,\quad\text{for }\xi\otimes u\in \mathcal{H}\otimes \CC^2.\]
Then $(\mathcal{A}, \mathcal{H}\otimes {\mathbb{C}}^2, D)$ is an odd spectral triple of dimension 3 of the QHMs.
\end{thm}

The argument will be split up into a sequence of lemmatas and propositions, which will be the content of the remainder of this section. Particular focus is put on identifying a suitable domain for the Dirac operator $D$ on which it is a self-adjoint operator.

{\bf Step 1: Identifying a domain for $D$}\\
First we transform the boundary conditions defined by using $\rho$ to periodic boundary conditions. Set $M_1:L^2(\mathbb{T}^2\times \mathbb{Z})\to L^2([0,1]\times \mathbb{T}\times \mathbb{Z})$ to be the operator
\[M_1\phi(x,y,p)=e(cp(y-p\nu)x)\phi(x,y,p).\]
The derivations $\tilde{\delta}_X,\,\tilde{\delta}_Y$, and $\tilde{\delta}_Z$ after conjugation by $M_1$ become
\begin{align*}
    &\delta_X'=M_1^{-1}\delta_XM_1=-2\pi icp^2\mu -\frac{d}{dy},\\
    &\delta_Y'=M_1^{-1}\delta_YM_1=-2\pi icp(y-p\nu) -\frac{d}{dx},\\
    &\delta_Z'=M_1^{-1}\delta_ZM_1=2\pi ip\,.
\end{align*}
Hence we need to find a domain where the operator  \[D'=(M_1\otimes\Id_{\CC^2})^{-1}D(M_1\otimes\Id_{\CC^2}),\]
is self-adjoint.
Using $M_1$ to transform the suggested domain of $D$, we get that the domain of $D'$ should satisfy
\begin{enumerate}[i)]
    \item $\Phi(\cdot,\cdot,p)\in H^1( \mathbb{T}^2)$ for all $p\in \ZZ$,
    \item $\Phi(0,y,p)= \Phi(1, y, p)$,
    \item $p\Phi$ and $\partial_y(e(cp(y-p\nu)x+\lfloor cp^2\mu\rfloor y)\Phi)$ and $\partial_x(e(cp(y-p\nu)x)\Phi)$ are in $\mathcal{H}$.
\end{enumerate}

Note that the operator $D'$ can be viewed as a family of operator $D'=\{D_p\}_{p\in \mathbb{Z}}$, where for $p\in \mathbb{Z}$ the operator $D_p$ is defined on $L^2(\mathbb{T}^2)\otimes\mathbb{C}^2$ by 
\begin{align*}
D_p \Phi(x,y)\otimes v &=-i\frac{\partial\Phi(x,y)}{\partial y}\otimes \sigma_X(v)-i\frac{\partial\Phi(x,y)}{\partial x}\otimes \sigma_Y(v)+M_{2\pi c p^2\mu}\Phi(x,y)\otimes \sigma_X(v)\\
&\quad+M_{2\pi pc(y-p\nu)}\Phi(x,y)\otimes \sigma_Y(v)- M_{2\pi p}\Phi(x,y)\otimes \sigma_Z(v).
\end{align*}
Consequently, the question of the domain for $D^\prime$ is reduced to one for $D_p$.
\begin{lem}
The domain of the operator $D_p$ is given by $\mathrm{Dom}(D_p)=H^1(\mathbb{T}^2)\otimes \mathbb{C}^2$.
\end{lem}
\begin{proof}
Denote by $S_p$ the bounded self-adjoint operator
\begin{equation*}
S_p \Phi(x,y)\otimes v =M_{2\pi c p^2\mu}\Phi(x,y)\otimes \sigma_X(v)+M_{2\pi pc(y-p\nu)}\Phi(x,y)\otimes \sigma_Y(v)- M_{2\pi p}\Phi(x,y)\otimes \sigma_Z(v),
\end{equation*}
and let $T_p$ be the unbounded symmetric operator
   \[T_p(\phi\otimes v)=-i\frac{\partial \phi}{\partial y}\otimes \sigma_X(v)-i\frac{\partial \phi}{\partial x}\otimes \sigma_Y(v)=\begin{pmatrix}-v_2\frac{d\phi}{dz}\\ v_1\frac{d\phi}{d\bar{z}}\end{pmatrix}.\]
Then since $D_p=T_p+S_p$, we have that if $T_p$ is a self-adjoint operator then $D_p$ are a self-adjoint operators on the same domain, after an application of \cite[Lemma 3.27]{Bo20}. In other words, if we can show that $(T_p, \mathrm{Dom}(T_p))$ is self-adjoint then $D_p$ is self-adjoint on $\mathrm{Dom}(T_p)=\mathrm{Dom}(D_p)$.

   {\bf Claim:} The operator $T_p$ with domain $\mathrm{Dom}(T_p)=H^1(\mathbb{T}^2)\otimes \mathbb{C}^2$ is self-adjoint.
\\
Recall that a symmetric operator $(V, \mathrm{Dom}(V))$ is self-adjoint if and only if the operators $(V+i\,\mathrm{Id}_{L^2(\TT^2)\otimes\CC^2}, \mathrm{Dom}(V))$ and $(V-i\,\mathrm{Id}_{L^2(\TT^2)\otimes\CC^2},\mathrm{Dom}(V))$ are surjective, see \cite[Thm.~3.29]{Bo20}.

Let us write a function $\phi\in H^1(\mathbb{T}^2)\otimes \mathbb{C}^2$ with respect to the Fourier basis:
\[\phi(x,y)\otimes v=\sum_{m,n\in \mathbb{Z}}a_{1,m,n}e^{2\pi i(xm+yn)}e_1+a_{2,m,n}e^{2\pi i(xm+yn)}e_2,\]
where $e_1$ and $e_2$ are the standard basis vectors for $\mathbb{C}^2$.
Then the operator $T_p+i\,\mathrm{Id}_{L^2(\TT^2)\otimes\CC^2}$ acts on $\phi$ by 
\begin{align*}
   (T_p+i\,\mathrm{Id}_{L^2(\TT^2)\otimes\CC^2})\phi(x,y)\otimes v&=\sum_{m,n\in \mathbb{Z}}(2\pi (n-im)a_{2,m,n}+ia_{1,m,n})e^{2\pi i(xm+yn)}e_1\\
   &\quad +\sum_{m,n\in \mathbb{Z}}(2\pi (n+im)a_{1,m,n}+ia_{2,m,n})e^{2\pi i(xm+yn)}e_2
\end{align*}
Since an arbitrary element in $\xi\otimes w\in L^2(\mathbb{T}^2)\otimes\mathbb{C}^2$ can be written as
   \[\xi\otimes w=\sum_{m,n\in \mathbb{Z}}b_{1,m,n}e^{2\pi i(xm+yn)}e_1+b_{2,m,n}e^{2\pi i(xm+yn)}e_2,\]
   for $T_p+i\,\mathrm{Id}_{L^2(\TT^2)\otimes\CC^2}$ to be surjective we need to find $a_{1,n,m}$ and $a_{2,n,m}$ satisfying 
\begin{multline*}
   \sum_{m,n\in \mathbb{Z}}(2\pi (n-im)a_{2,m,n}+ia_{1,m,n})e^{2\pi i(xm+yn)}e_1+(2\pi (n+im)a_{1,m,n}+ia_{2,m,n})e^{2\pi i(xm+yn)}e_2\\
   =\sum_{m,n\in \mathbb{Z}}b_{1,m,n}e^{2\pi i(xm+yn)}e_1+b_{2,m,n}e^{2\pi i(xm+yn)}e_2.
\end{multline*}
Solving for $a_{1,m,n}$ and $a_{2,m,n}$ one can see that
\[a_{1,n,m}=\frac{-ib_{1,n,m}+2\pi (n-mi)b_{2,n,m}}{4\pi^2(m^2+n^2)+1}\]
and
\[a_{2,n,m}=\frac{-ib_{2,n,m}+2\pi (n+im)b_{1,n,m}}{4\pi^2(m^2+n^2)+1}.\]
Now, the Sobolev norm of $\phi\otimes v$ is bounded by
\begin{align*}
\|\phi\otimes v\|_{H^1(\mathbb{T}^2)\otimes\mathbb{C}^2}&=|b_{1,0,0}|^2+|b_{2,0,0}|^2\\
&\quad+\sum_{m,n\in \mathbb{Z}}(m^2+n^2)\left|\frac{-ib_{1,n,m}+2\pi (n-mi)b_{2,n,m}}{4\pi^2(m^2+n^2)+1}\right|^2\\
&\quad+\sum_{m,n\in \mathbb{Z}}(m^2+n^2)\left|\frac{-ib_{2,n,m}+2\pi (n+mi)b_{1,n,m}}{4\pi^2(m^2+n^2)+1}\right|^2\\
&\le C\|\xi\otimes w\|_{L^2(\mathbb{T}^2)\otimes\mathbb{C}^2}<\infty.
\end{align*}
This yields the desired claim that $T_p+i\,\mathrm{Id}_{L^2(\TT^2)\otimes\CC^2}:\mathrm{Dom}(D_p)\to L^2(\mathbb{T}^2)\otimes\mathbb{C}^2$ is surjective.

The argument for the surjectivity of $T_p-i\,\mathrm{Id}_{L^2(\TT^2)\otimes\CC^2}$ goes along the same lines. The desired assertion follows by the boundedness of $S_p$.
\end{proof}
Let us now identify a domain where the operator $D'$ is self-adjoint. 
A necessary condition for the function  $\phi\otimes w\in \mathcal{H}\otimes \CC^2$ to be in the domain of $D'$ is that \[\sum_{p\in \ZZ}\|D_p\phi(\cdot,\cdot,p)\otimes w\|_{L^2(\TT^2)\otimes \CC^2}^2<\infty\]
otherwise $D'\phi\otimes w\not\in L^2(\mathbb{T}^2\times \mathbb{Z})\otimes \CC^2$. Let us see what this amounts to more concretely.

For $v=v_1e_1+v_2e_2\in \CC^2$ we have the identities  
\begin{align*}
\langle\sigma_X(v),\sigma_Y(v)\rangle_{\CC^2}&=i(|v_2|^2-|v_1|^2)=-\langle\sigma_Y(v),\sigma_X(v)\rangle_{\CC^2}, \\
\langle\sigma_X(v),\sigma_Z(v)\rangle_{\CC^2}&=2i\mathrm{Im}(v_2\bar{v_1})=-\langle\sigma_Z(v),\sigma_X(v)\rangle_{\CC^2}, \\ 
\langle\sigma_Y(v),\sigma_Z(v)\rangle_{\CC^2}&=-2i\mathrm{Re}(v_2\bar{v_1})=-\langle\sigma_Z(v),\sigma_Y(v)\rangle_{\CC^2}.
\end{align*}
Writing our the norm gives
\begin{align*}
\sum_{p\in \ZZ}\|D_p\phi(\cdot,\cdot,p)\otimes v\|_{L^2(\TT^2)\otimes \CC^2}^2
&=\|\delta_X'\phi\|_{L^2(\TT^2\times \ZZ)}^2\|v\|_{\CC^2}^2+\|\delta_Y'\phi\|_{L^2(\TT^2\times \ZZ)}^2\|v\|_{\CC^2}^2+\|\delta_Z'\phi\|_{L^2(\TT^2\times \ZZ)}^2\|v\|_{\CC^2}^2\\
&\quad+\left(\langle i\delta_X'\phi,i\delta_Y'\phi\rangle_{L^2(\TT^2\times \ZZ)}-\langle i\delta_Y'\phi,i\delta_X'\phi\rangle_{L^2(\TT^2\times \ZZ)}\right)\langle\sigma_X(v),\sigma_Y(v)\rangle_{\CC^2}\\
&\quad+\left(\langle i\delta_Y'\phi,i\delta_Z'\phi\rangle_{L^2(\TT^2\times \ZZ)}-\langle i\delta_Z'\phi,i\delta_Y'\phi\rangle_{L^2(\TT^2\times \ZZ)}\right)\langle\sigma_Y(v),\sigma_Z(v)\rangle_{\CC^2}\\
&\quad+\left(\langle i\delta_X'\phi,i\delta_Z'\phi\rangle_{L^2(\TT^2\times \ZZ)}-\langle i\delta_X'\phi,i\delta_Y'\phi\rangle_{L^2(\TT^2\times \ZZ)}\right)\langle\sigma_X(v),\sigma_Z(v)\rangle_{\CC^2}.
\end{align*}
Using that $i\delta_X, i\delta_Y,$ and $i\delta_Z$ are symmetric, we get that 
\begin{equation*}
    \langle i\delta_Y'\phi,i\delta_Z'\phi\rangle_{L^2(\TT^2)}-\langle i\delta_Z'\phi,i\delta_Y'\phi\rangle_{L^2(\TT^2)}=\langle i\delta_X'\phi,i\delta_Z'\phi\rangle_{L^2(\TT^2)}-\langle i\delta_X'\phi,i\delta_Y'\phi\rangle_{L^2(\TT^2)}=0.
\end{equation*}
For the last cross term, we have that if $\phi$ is twice differentiable, then
\begin{equation*}
    \langle i\delta_X'\phi,i\delta_Y'\phi\rangle_{L^2(\TT^2)}-\langle i\delta_Y'\phi,i\delta_X'\phi\rangle_{L^2(\TT^2)}=\langle (-\delta_X'\delta_Y'+\delta_X'\delta_Y')\phi,\phi\rangle_{L^2(\TT^2)}=-\langle c\delta_Z'\phi,\phi\rangle_{L^2(\TT^2)}.
\end{equation*}
By using that twice differentiable functions are dense in $H^1(\TT^2)$ we get that the same identity holds for all $\phi$.
Finally, completing squares gives 
\begin{align*}
 \sum_{p\in \ZZ}\|D_p\phi(\cdot,\cdot,p)\otimes v\|_{L^2(\TT^2)\otimes \CC^2}^2
&=\|\delta_X'\phi\|_{L^2(\TT^2\times \ZZ)}^2\|v\|_{\CC^2}^2+\|\delta_Y'\phi\|_{L^2(\TT^2\times \ZZ)}^2\|v\|_{\CC^2}^2\\
&\quad+\|(2\pi p-c/2)\phi\|_{L^2(\TT^2\times \ZZ)}^2|v_1|^2+\|(2\pi p+c/2)\phi\|_{L^2(\TT^2\times \ZZ)}^2|v_2|^2\\
&\quad-\frac{c^2}{4}\|\phi\|^2_{L^2(\TT^2\times \ZZ)}\|v\|_{\CC^2}^2.
\end{align*}
Thus, this expression is finite if $\phi$ is an element of the suggested domain.

We are only left to show that the above expression can not be finite for any $\phi$ outside the suggested domain.
Denote by $M_2$ the operator given by $M_2\phi=e(\lfloor cp^2\nu\rfloor x-\lfloor cp^2\mu\rfloor y)\phi$.
By an application of  Young's inequality and the reverse triangle inequality
\begin{equation}\label{eq:Young}
\|u+w\|^2\ge (1-\epsilon)\|u\|+\left(1-\frac{1}{\epsilon}\right)\|v\|^2
\end{equation}
we deduce a lower bound for the operator norm of  
\begin{align*}
    \|(M_2\otimes \Id)^{-1}D'(M_2\otimes \Id)\phi\otimes v\|_{L^2(\TT^2\times \ZZ)\otimes \CC^2}^2&\ge \|v\|_{\CC^2}^2\Big(\big\|2\pi(cp^2\mu-\lfloor cp^2\mu\rfloor)\phi-i\frac{d\phi}{dy}\big\|_{L^2(\TT^2\times \ZZ)}^2\\
    &\quad+\big\|2\pi( cpy+cp^2\nu-\lfloor cp^2\nu\rfloor)\phi-i\frac{d\phi}{dx}\big\|_{L^2(\TT^2\times \ZZ)}^2\\
    &\quad-2c^2\|\phi\|_{L^2(\TT^2\times \ZZ)}^2\Big)+\|(2\pi p-c)\phi\|_{L^2(\TT^2\times \ZZ)}^2|v_1|^2\\
    &\quad+\|(2\pi p+c)\phi\|_{L^2(\TT^2\times \ZZ)}^2|v_2|^2.
\end{align*}
By Young's inequality \eqref{eq:Young} we have that 
\[\big\|2\pi(cp^2\mu-\lfloor cp^2\mu\rfloor)\phi-i\frac{d\phi}{dy}\big\|_{L^2(\TT^2\times \ZZ)}^2\ge \frac{1}{8}\big\|\frac{d}{dy}\phi\big\|_{L^2(\TT^2\times \ZZ)}^2-\frac{4\pi^2}{7}\|\phi\|_{L^2(\TT^2\times \ZZ)}^2,\]
\begin{multline*}
\big\|2\pi cpy\phi+2\pi(cp^2\nu-\lfloor cp^2\nu\rfloor)\phi-i\frac{d\phi}{dy}\big\|_{L^2(\TT^2\times \ZZ)}^2\ge \frac{1}{8(c^2+1)}\big\|\frac{d\phi}{dx}\big\|_{L^2(\TT^2\times \ZZ)}^2-\frac{4\pi^2}{7}\|\phi\|_{L^2(\TT^2\times \ZZ)}^2\\
-\frac{1}{8}\|2\pi p\phi\|_{L^2(\TT^2\times \ZZ)}^2
\end{multline*}
and 
\[\|(2\pi p\pm c/2)\phi\|_{L^2(\TT^2\times \ZZ)}^2\ge \frac{3}{4}\|2\pi p\phi\|_{L^2(\TT^2\times \ZZ)}^2-\frac{3c^2}{4}\|\phi\|_{L^2(\TT^2\times \ZZ)}^2.\]
Collecting all these inequalities demonstrates that 
\begin{multline} \label{eq:lower_bound_D}
    \|(M_2\otimes \Id)^{-1}D'M_2\otimes \Id\phi\otimes v\|_{L^2(\TT^2\times \ZZ)\otimes\CC^2}^2\ge \|v\|_{\CC^2}^2\bigg(\frac{1}{8(c^2+1)}\Big(\big\|\frac{d\phi}{dy}\big\|_{L^2(\TT^2\times \ZZ)}^2\\+\big\|\frac{d\phi}{dx}\big\|_{L^2(\TT^2\times \ZZ)}^2 +\|2\pi p\phi\|_{L^2(\TT^2\times \ZZ)}^2\Big)-\Big(\frac{8\pi^2}{7}+c^2\Big)\|\phi\|_{L^2(\TT^2\times \ZZ)}^2\bigg).
\end{multline}
In other words, a function $\Phi$ is in $\mathrm{Dom}(D')$ if it satisfies: 
\begin{enumerate}[(i)]
    \item $\Phi(\cdot,\cdot,p)\in H^1(\mathbb{T}^2)$,
    \item $p\Phi$, $\partial_y(e(\lfloor cp^2\mu\rfloor y)\Phi)$ and $\partial_x(e(\lfloor cp^2\nu\rfloor x)\Phi)$ are in $L^2(\mathbb{T}^2\times \mathbb{Z})$.
\end{enumerate}
These conditions are equivalent to the ones given for $D'$ above.

Note that on $\operatorname{Dom}(D)$, $D=T+S$, where $\operatorname{Dom}(T)=\operatorname{Dom}(D) \subseteq \operatorname{Dom}(S)$ and $T, S$ are given by
\[
T=-i\frac{\partial}{\partial y} \otimes \sigma_X(u)-i\frac{\partial}{\partial x} \otimes \sigma_Y(u)-2\pi M_p\otimes \sigma_Z(u),
\]
\[
S=-2\pi M_{cp(x-p\mu)}\otimes \sigma_X(u).
\]
where $M$ is a multiplication operator. The arguments given in the proof of Proposition~9 of \cite{ChS} still holds in our setting. The operators $T$ and $S$ are self-adjoint on their respective domains, and $T$ has compact resolvent. Since $S$ is relatively bounded with respect to $T$, Rellich's lemma implies that $S$ has a compact resolvent, too.
\\~~\\
{\bf Step 2: Boudedness of $[D,a]$}

Take an $a\in \DD$. We want to establish that 
\[[\tilde{\pi}(a),D]=i[\tilde{\pi}(a), \delta_X]\otimes \sigma_X+i[\tilde{\pi}(a), \delta_Y]\otimes \sigma_Y+i[\tilde{\pi}(a), \delta_Z]\otimes \sigma_Z\]
is bounded.
Let us first compute the commutators of $\tilde{\pi}(a)$ with $\delta_X,\delta_Y$ and $\delta_Z$:
\[ [\tilde{\pi}(a), \delta_Y]=-\tilde{\pi}(\delta_Y a),\quad[\tilde{\pi}(a), \delta_Z]=-\tilde{\pi}(\delta_Z a),\quad\text{and}\quad  [\tilde{\pi}(a), \delta_X]=-\tilde{\pi}(\delta_Xa).\] 
Hence we have that 
\begin{equation}\label{eq:commutator}
[\tilde{\pi}(a),D]=-i\tilde{\pi}\left(\delta_X a\right)\otimes \sigma_X-i\tilde{\pi}(\delta_Y a)\otimes \sigma_Y-i\tilde{\pi}(\delta_Z a)\otimes \sigma_Z=-\tilde{\pi}(Da).
\end{equation}
Since we know that $\tilde{\pi}:\mathcal{A}\to \mathcal{B}(L^2([0,1]\times \mathbb{T}\times \mathbb{Z}))$, the result follows from $D(a)\in \mathcal{A}$.

\vspace{0.5cm}

{\bf Step 3: The spectral triple $(\mathcal{A}, \mathcal{H}\otimes {\mathbb{C}}^2, D)$ has dimension $3$.}

Before we discuss the behavior of the eigenvalues of $D^2$, we recall that the eigenvalues of the Laplacian on $\mathbb{T}^3$ can be written as 
\[\lambda_k=\max_{\substack{V\subset H^2(\mathbb{T}^3)\\ \mathrm{dim}(V)=k}}\min_{u\in V\setminus \{0\}}\frac{\|\nabla u\|^2}{\|u\|^2},\]
where $\nabla$ denotes the gradient. By Weyl's law formula  we deduce that 
\[\lim_{k\to \infty}\frac{\lambda_k}{k^{2/3}}=C.\]

This control of the eigenvalue behavior of the Laplacian suffices to find the behavior of eigenvalues of $|D|^{-3}$.

%Since $D$ is self-adjoint we have that $D^2$ is a positive operator. 
Denote by $\mathcal{F}:L^2(\mathbb{T}^2\otimes \mathbb{Z})\to L^2(\mathbb{T}^3)$
the inverse Fourier transform in the $p$ variable, which takes $M_{2\pi i p}$ to $\partial_z$. The lower bound \eqref{eq:lower_bound_D} combined with the min-max theorem, see \cite[Thm.~5.15]{Bo20} yields that 
\begin{align*}
\mu_k&=\max_{\substack{V\subset \mathrm{Dom}(D^2)\\ \mathrm{dim}(V)=k}}\min_{\phi\otimes v\in V\setminus \{0\}}\frac{\langle D^2 \phi\otimes v, \phi\otimes v\rangle }{\|\phi\otimes v\|^2}\\
&=\max_{\substack{V\subset \mathrm{Dom}(D^2)\\ \mathrm{dim}(V)=k}}\min_{\phi\otimes v\in V\setminus \{0\}}\frac{\langle D \phi\otimes v, D\phi\otimes v\rangle }{\|\phi\|^2\| v\|^2}\\
&\ge\max_{\substack{V\subset H^1(\mathbb{T}^3)\otimes \mathbb{C}^2\\ \mathrm{dim}(V)=k}}\min_{\phi\otimes v\in V\setminus \{0\}}\frac{\frac{1}{8(c^2+1)}\big(\|\frac{d\phi}{dy}\|^2+\|\frac{d\phi}{dx}\|^2+\|\frac{d\phi}{dz}\|^2 \big)-(\frac{8\pi^2}{7}+c^2)\|\phi\|^2}{\|\phi\|^2}\\
&\ge\frac{1}{8(c^2+1)}\max_{\substack{V\subset H^1(\mathbb{T}^3)\otimes \mathbb{C}^2\\ \mathrm{dim}(V)=k}}\min_{\phi\otimes v\in V\setminus \{0\}}\frac{\|\nabla \phi\|^2}{\|\phi\|^2}-\Big(\frac{8\pi^2}{7}+c^2\Big),
\end{align*}
which proves that for large $k$, the eigenvalue $\mu_k\ge C(c,\mu) k^{2/3}$. We set $|D|^{-3}$ to be zero on $\ker(D)$. In particular, we have that $|D|^{-3}$ has eigenvalues that go to zero faster than $\frac{1}{k}$. Hence the sum of the first $k$-th largest eigenvalues grows slower than $\log(k)$, which finishes our argument.

\section{Sigma model on the QHMs}
In this section, we derive the sigma-model energy functional $S$ defined on the set of projections of $\DD$ for the spectral triples introduced in \cite{ChS} and \cite{ChG}. There is a natural way to obtain a sigma-model energy functional for spectral triples \cite[Section~4.1]{MaRos} for a given spectral triple $(\mathcal{A}, \mathcal{H},D)$. Then they define a positive Hochschild 2-cocyle $\psi_2$ on $\mathcal{A}$ by
\[
\psi_2(a_0,a_1, a_2)=\Tr\Big( (\Id_{\CC^2}+\gamma)a_0 [D,a_1] [D,a_2]\Big)
\]
where $\gamma=\sigma_Z$ is a grading operator on $\mathcal{H}$, and where $\Tr$ is the canonical trace of matrices. Since we do not work with an even spectral triple we will remove the grading operator in the definition, however the definitions end up being equivalent.

Recall that an element in $P\in \DD$ is said to be a \textit{projection} if we have that \[P=P^2=P^*.\]
Using this relation we have that the action of $P$ on $L^2([0,1]\times \TT\times \ZZ)$ is an orthogonal projection. With respect to the spectral triple $(\mathcal{A}, L^2([0,1]\times\TT\times \ZZ)\otimes \CC^2, D)$ define a \textit{Hochschild 2-chain} by \[\psi_2(a,a_1,a_2)=\Tr(a[D,a_1][D,a_2]),\]
where $\Tr$ denotes the matrix trace.
We say that $(a,a_1,a_2)$ is a \textit{Hochschild 2-cocyle} with respect to $D$ if the boundary map
\[b(\psi_2(a,a_1,a_2))=\Tr(aa_1[D,a_2]-a[D,a_1a_2]+a_2a[D,a_1])=0.\]
Notice that since the Pauli matrices have trace $0$ the boundary map is always zero.

Let $\mathcal{B}$ be a $C^*$-algebra and consider a set of homomorphisms $\psi$ from a smooth subalgebra of $\mathcal{B}$ with target $\mathcal{A}$. Then the general sigma-model energy functional $S$ associated to the spectral triple is given by
\[
S(\phi)=\phi^\ast(\psi),
\]
where $\phi^*(\psi)$ is a positive Hochschild 2-cocycle on $\mathcal{B}$. 

We are interested in the case where $\mathcal{B}$ is $\mathbb{C}$ and $\phi$ is $\tau_D$ and $\psi$ is $\psi_2$. Then the energy functional associated to $D$ and $\phi$ defined in \cite[Section 4]{MaRos} is given in the next proposition.
%{\color{blue} (Franz, could you put some more explanation or references for $\psi^\ast=\tau_D$?)}

\begin{prop} \label{prop:S}

For given the basis $\{X,Y,Z\}$ of the Heisenberg Lie algebra $\mathfrak{h}$ with $[X,Y]=cZ$,
the action functional on the set of projections of $\DD$ of the sigma model is given as follow:
\[\begin{split}
S(p)&:=-\tau_D(\psi_2(\Id, P,P))=2 \,\tau_D\big((\delta_X P)^2 + (\delta_Y P)^2+ (\delta_Z P)^2\big).
\end{split}\]
\end{prop}

\begin{proof}
Let $D$ denote the Dirac operator given in Theorem \ref{thm:spectral tripple} which is given by
\[
D(f\otimes u)=\sum_{j=X,Y,Z} i \, \tilde{\delta_j}(f)\otimes \sigma_j(u),
\]
where $\sigma_j$'s are the Pauli spin matrices given in \eqref{eq:Pauli_matrix}.

As we saw in \eqref{eq:commutator}, for any $P\in \DD$ we have that
\[
    [D,\tilde{\pi}(P)]=\tilde{\pi}(DP).\]
%For $a\in \DD$, we first compute $[D,a]$ as follows:
%\[\begin{split}
%&[D,a](f\otimes u)=Da(f\otimes u)-aD(f\otimes u)=D(af\otimes u)-aD(f\otimes u)\\
%&=\sum_{j=X,Y,Z} i\delta_j(af)\otimes \sigma_j(u)-a\sum_{j=X,Y,Z} i\delta_j(f)\otimes \sigma_j(u)\\
%&=\sum_{j=X,Y,Z} i(\delta_j(a) f+a\delta_j(f))\otimes \sigma_j(u)-a\sum_{j=X,Y,Z} i\delta_j(f)\otimes \sigma_j(u)\\
%&=\sum_{j=X,Y,Z} i(\delta_j(a) f) \otimes \sigma_j(u).
%\end{split}\]
Using that $\tilde{\pi}$ is a representation on $L^2([0,1]\times\mathbb{T}\times\mathbb{Z})$ we get that
\[
    [D,\tilde{\pi}(P)] [D,\tilde{\pi}(P)]=\tilde{\pi}(DP)\tilde{\pi}(DP)=\tilde{\pi}((DP)^2).\]
Computing $(DP)^2$ we get
\[\begin{split}
(DP)^2&= \left(\sum_{j=X,Y,Z} i(\delta_j(P)) \otimes \sigma_j\right)^2\\
&=\sum_{k=X,Y,Z} i\delta_k(P) (i\delta_X(P))\otimes \sigma_k(\sigma_X) + \sum_{k=X,Y,Z} i\delta_k(P) (i\delta_Y(P))\otimes \sigma_k(\sigma_Y) \\
&\quad\quad\quad + \sum_{k=X,Y,Z} i\delta_k(P) (i\delta_Z(P))\otimes \sigma_k(\sigma_Z)\\
&=-(\delta_X(P))^2 \otimes \sigma_X^2 -\delta_X(P)\delta_Y(P) \otimes \sigma_X\sigma_Y -\delta_X(P)\delta_Z(P)  \otimes \sigma_X \sigma_Z \\
&\quad-\delta_Y(P)\delta_X(P)\otimes \sigma_Y\sigma_X-(\delta_Y(P))^2\otimes \sigma^2_Y -\delta_Y(P)\delta_Z(P) \otimes \sigma_Y\sigma_Z\\
&\quad -\delta_Z(P)\delta_X(P)  \otimes \sigma_Z\sigma_X-\delta_Z(P)\delta_Y(P)\sigma_Z\sigma_Y-(\delta_Z(P))^2 \otimes \sigma_Z^2\\
&=-((\delta_X(P))^2 +(\delta_Y(P))^2+(\delta_Z(P))^2)\otimes \id -i[\delta_X(P),\delta_Y(P)] \otimes \sigma_Z -i[\delta_Z(P),\delta_X(P)  ]\otimes \sigma_Y \\
&\quad -i[\delta_Y(P),\delta_Z(P)] \otimes \sigma_X.\\
\end{split}\]
Noting that $\tau_D([\delta_i(P),\delta_j(P)])=0$
for all $i,j\in \{X,Y,Z\}$ gives that all the commutator terms will vanish. Hence we will continue the computation with only the term
%{\color{red}After taking trace, then we get trace of the previous is the same as the trace of the following.}
\[
-((\delta_X(P))^2+(\delta_Y(P))^2+(\delta_Z(P))^2) \otimes \begin{pmatrix} 1 & 0 \\ 0 & 1 \end{pmatrix} .
\]
Thus 
\[\begin{split}
\psi_2(\Id, P,P)&= \operatorname{Tr}(\Id [\tilde{D},P][\tilde{D},P])=\operatorname{Tr}\left(  \sum_{j=X,Y,Z} \delta_j(P))^2 \otimes \Id_{\CC^2}\right)\\
&=-2 ((\delta_X P)^2 + (\delta_Y P)^2+ (\delta_Z P)^2).
\end{split}
\]
Therefore we have
\[
S(P):=-\tau_D(\psi_2(\Id, P, P))=2\tau_D((\delta_X P)^2 + (\delta_Y P)^2+ (\delta_Z P)^2).\qedhere
\]

\end{proof}

The \textit{topological charge} of a projection $P\in \DD$ is given by
\[
 c_\nabla(P):=c_{XY}(P)+c_{XZ}(P)+c_{YZ}(P), 
\]
where 

\[
c_{VW} (P) =\frac{1}{2\pi i}\,  \tau_D (P(\delta_V P \,\delta_W P-\delta_W P \,\delta_V P)) \quad \text{for $V,W\in \mathfrak{h}$.}
\]
Then one can show that the topological charge gives a lower bound of the action functional $S$ as follows.

\begin{lem}\label{lem:1} (C.f. Proposition~2.1 of \cite{DLL})
For a projection $P\in \DD$, we have that
\[
S(P) \ge |c_{XY}(P)| + |c_{YZ}(P)|+|c_{XZ}(P)| \ge  |c_\nabla(P)|.
\]
\end{lem}
\begin{proof}
By Proposition~\ref{prop:S} above, we know that
\[
S(P)= 2 \,\tau_D\big((\delta_X P)^2 + (\delta_Y P)^2+ (\delta_Z P)^2\big)
\]
We decompose $S(P)$ as follows: 
\[S(P)=\tau_D \Big((\delta_X P)^2 + (\delta_Y P)^2\Big)+\tau_D \Big((\delta_X P)^2 + (\delta_Z P)^2\Big) +\tau_D\Big((\delta_Y P)^2+(\delta_Z P)^2\Big).\]

Subsequent application of Proposition~2.1 of \cite{DLL} to each of the terms yields that \[S(P)\ge |c_{XY}(P)|+|c_{XZ}(P)|+|c_{YZ}(P)| \ge |c_\nabla(P)|.\qedhere\]
\end{proof}

By a similar argument to the one given in Proposition~3.2 of \cite{DLL}, one can verify the following: In case that $P$ is of the form $\langle R, R \rangle_R^D$ for an appropriately chosen $R\in\Xi$ one can express topological charges in terms of  curvatures $\Theta_\nabla(X,Y)$, $\Theta_\nabla(X,Z)$, and $\Theta_\nabla(Y,Z)$.

\begin{thm} \label{prop:1} (C.f. \cite[Proposition~3.2]{DLL})
Assume that $R\in \Xi$ is such that $\langle R, R \rangle_R^D=P$ is a projection and $\langle R, R \rangle_L^E=\Id_E$. Let $\nabla$ be a compatible connection and denote the curvature of $\nabla$ by $\Theta_\nabla$.
Then the topological charge of the projection $P$ is given by
\[
c_\nabla(P)=-\frac{1}{2\pi i} \tau_E(\langle R, (\Theta_\nabla(X,Y) +\nabla_{[X,Y]} +\Theta_\nabla(X,Z)+\Theta_\nabla(Y,Z))\cdot R\rangle_L^E).
\]
\end{thm}
\begin{proof}
We adapt the arguments of \cite[Proposition~3.2]{DLL}, where the topological charge of a projection and the latter is expressed in terms of the curvature of a connection on a projective module. Since $P=\langle R, R \rangle_R^D$ we can express \[c_{XY}=\frac{1}{2\pi i}\tau_D\Big(P\Big(\delta_X(P)\delta_Y(P)-\delta_Y(P)\delta_X(P)\Big)\Big)\] on the form 
\[-\frac{1}{2\pi i} \tau_E(\langle R, ([\nabla_X, \nabla_Y] \cdot R)\rangle_L^E),\] where we use that $\nabla$ is compatible with the both $\langle\cdot, \cdot\rangle_{R}^D$ and $\langle\cdot, \cdot\rangle_{L}^E$ on $\Xi$. 
If we proceed in the same way for $c_{YZ}(P)$ and $c_{XZ}(P)$, then we obtain the desired identity for $c_\nabla(P)$.
\end{proof}

Note that the functions $R$ are exactly the ones used to define the Grassmannian connection given in \eqref{eq:Gr-conn}.

We show below that the lower bound of the functional $S$ depends on the structure of the compatible connection. This result is quite unexpected and is due to the fact that such connections are non-constant. Hence, one does not observe this kind of behavior in the case of noncommutative tori where all compatible connections on Heisenberg modules have constant curvature.

\begin{thm}
Let $R\in \Xi$ be such that $\langle R, R \rangle_R^D=P$ is a projection and $\langle R, R \rangle_L^E=\Id_E$. Let $\nabla$ be a Yang-Mills connection on $\Xi$ with curvature 
\[
\Theta_\nabla(X,Y)=0, \quad \Theta_\nabla(X,Z)=0, \quad \text{and} \quad \Theta_\nabla(Y,Z)=\frac{\pi i}{\mu} \Id_E.
\]
 Then
\[
S(Q) \ge  \Big| \frac{c}{2 \pi i} \, \tau_E(\langle R, \nabla_Z(R) \rangle_L^E) -1 \Big|.
\]

\end{thm}

\begin{proof}
By Lemma~\ref{lem:1} and Proposition~\ref{prop:1}, we get 
\[
S(Q)\ge |c_\nabla(Q)| = \Big| -\frac{1}{2\pi i} \tau_E(\langle R, (\Theta_\nabla(X,Y) +\nabla_{[X,Y]} +\Theta_\nabla(X,Z)+\Theta_\nabla(Y,Z))\cdot R\rangle_L^E) \Big|
\]
Since $[X,Y]=c\, Z$ and $\Theta_\nabla(X,Y)=0$,  $\Theta_\nabla(X,Z)=0$ and $\Theta_\nabla(Y,Z)=\frac{\pi i}{\mu} \Id_E$, we have 
\[\begin{split}
|c_\nabla(Q)| &= \Big| -\frac{1}{2\pi i} \tau_E(\langle R, (c \, \nabla_Z +\frac{\pi i}{\mu}\, \Id_E)\cdot R\rangle_L^E) \Big| \\
&=  \Big| -\frac{1}{2\pi i} \tau_E (\langle R,  c\, \nabla_Z (R)\rangle_L^E   + \langle R, \frac{\pi i}{\mu} R \rangle_L^E) \Big| \\
&= \Big| -\frac{1}{2\pi i} \tau_E (\langle R,  c\, \nabla_Z (R)\rangle_L^E  -\frac{1}{2 \pi i} \tau_E(-\frac{\pi i}{\mu} \langle R, R\rangle_L^E ) \Big| \\
&=  \Big| -\frac{c}{2\pi i} \tau_E (\langle R,   \nabla_Z (R)\rangle_L^E +\frac{1}{2\mu} \tau_E(\langle R, R \rangle_L^E) \Big|.
\end{split}\]
Since $\langle R, R \rangle_L^E=\Id_E=\delta_0(p)$ and $\tau_E(\Id_E)=2\mu$ (by using \eqref{eq:trace_E}), we have
\[
S(Q)\ge |c_\nabla(Q)| = \Big| -\frac{c}{2\pi i} \tau_E (\langle R,   \nabla_Z (R)\rangle_L^E +1\Big|,
\]
which gives the desired result.
\end{proof}

\begin{example}\label{Ex:1}
In this example, we compute the lower bound of $S(P)$ explicitly with the Yang-Mills connection $\nabla^0$ given in \eqref{eq:Lee_conn}, where $P$ is the projection with $P=\langle R, R \rangle_R^D$ and $\langle R, R \rangle_L^E=\Id_E$. In particular, $R$ can be set to the special function given in \cite{Kang1}.

As noted in Section \ref{sec:prelim}, the curvature of $\nabla^0$ is given by 
\[
\Theta_{\nabla^0}(X,Y)=0,\;\;\Theta_{\nabla^0}(X,Z)=0,\;\;\Theta_{\nabla^0}(Y,Z)=\frac{\pi i}{\mu}\Id_E,
\]
and $\nabla^0_Z$ is given by
\[
(\nabla^0_Z\xi)(x,y)=\frac{\pi i x}{\mu}\xi(x,y).
\]
Using the formula of $\tau_E$ in \eqref{eq:trace_E} and $E$-valued inner product in \eqref{E-value-inner}, we have
\begin{equation}\label{eq:trace_int}
\begin{split}
&-\frac{c}{2\pi i} \tau_E(\langle R, \nabla_Z^0(R)\rangle_L^E)\\
&= -\frac{c}{2\pi i} \int_0^{2\mu} \int_0^1 \langle R, \nabla^0_Z(R)\rangle_L^E (x,y,0)\, dx\, dy\\
&=-\frac{c}{2\pi i} \int_0^{2\mu} \int_0^1 \sum_{k\in \mathbb{Z}} \overline{R}(x-2 k \mu, y-2 k \nu) \frac{\pi i}{\mu} (x-2 k \mu) R(x-2 k \mu, y-2 k \nu) \, dx \, dy \\
&=-\frac{c}{2\pi i} \frac{\pi i}{\mu} \int_0^{2\mu} \int_0^1\sum_{k\in \mathbb{Z}} (x-2 k \mu) |R(x-2 k \mu, y-2 k \nu) |^2 \, dx\, dy \\
&=-\frac{c}{2\mu} \int_0^{2\mu} \int_0^1\sum_{k\in \mathbb{Z}} (x-2 k \mu) |R(x-2 k \mu, y-2 k \nu) |^2 \, dx\, dy \\
& =- \frac{c}{2\mu} \int_0^{2\mu} \int_0^1  x \sum_{k\in \mathbb{Z}}  |R(x-2 k \mu, y-2 k \nu) |^2 \, dx\, dy \\
& \quad +  \frac{c}{2\mu} \int_0^{2\mu} \int_0^1\sum_{k\in \mathbb{Z}} 2 k \mu |R(x-2 k \mu, y-2 k \nu) |^2 \, dx\, dy. 
\end{split}
\end{equation}

Since $\sum_{k\in \Z} |R(x-2 k \mu, y-2 k \nu) |^2 =1$ by (d-1) of \cite{Kang1}, the first integral above gives
\[
- \frac{c}{2\mu} \int_0^{2\mu} \int_0^1  x \sum_{k\in \mathbb{Z}}  |R(x-2 k \mu, y-2 k \nu) |^2 \, dx\, dy=-\frac{c}{2}.
\]

Hence we are only left with computing the value of the second integral. Beginning with rewriting the sum
\begin{equation}\label{eq:sum1}
\begin{split}
 & \sum_{k=-n}^n 2 k \mu |R(x-2 k \mu, y-2 k \nu) |^2 \\
 & =  \sum_{k=1}^n 2 k \mu |R(x-2 k \mu, y-2 k \nu) |^2 + \sum_{k=-n}^{-1} 2 k \mu |R(x-2 k \mu, y-2 k \nu) |^2.
\end{split}
\end{equation}

Recall the the summation by parts formula:
Let $\{a_n\}, \{b_n\}$ be sequences and let $S_N=\sum_{n=1}^N a_n$. Then for $0\le m \le n$ we have
\begin{equation}\label{eq:PS}
\sum_{k=m}^n a_k b_k = (S_n b_n-S_{m-1} b_m)+\sum_{k=m}^{n-1} S_k(b_k-b_{k+1}).
\end{equation}
By letting $a_k=|R(x-2 k \mu, y-2 k \nu) |^2$, $b_k=2 k \mu$, $S_N=\sum_{k=1}^n a_n$ and applying \eqref{eq:PS} with $m=1$ to the first summation of \eqref{eq:sum1}, we get
\[\begin{split}
&\sum_{k=1}^n 2 k \mu |R(x-2 k \mu, y-2 k \nu) |^2 \\
&=\sum_{k=1}^n  |R(x-2 k \mu, y-2 k \nu) |^2 \, (2 n \mu) + \sum_{k=1}^{n-1} \sum_{j=1}^k  |R(x-2 j \mu, y-2 j \nu) |^2 (2 k \mu-2 (k+1)\mu) \\
&=2 n \mu \sum_{k=1}^n  |R(x-2 k \mu, y-2 k \nu) |^2-2 \mu \sum_{k=1}^{n-1} \sum_{j=1}^k  |R(x-2 j \mu, y-2 j \nu) |^2. 
\end{split}\]

Similarly, by applying the same result to the second summation of \eqref{eq:sum1}, we obtain
\[\begin{split}
&\sum_{k=-n}^{-1} 2 k \mu  |R(x-2 k \mu, y-2 k \nu) |^2=-\sum_{k=1}^n 2 k \mu  |R(x+2 k \mu, y+2  k \nu) |^2 \\
&=-2 n \mu \sum_{k=1}^n  |R(x+2 k \mu, y+2  k \nu) |^2 + 2 \mu \sum_{k=1}^{n-1} \sum_{j=1}^k  |R(x+2 k \mu, y+2  k \nu) |^2.
\end{split}\]

%{\color{red} 
Now we apply the conditions on $R$ given in \eqref{eq:R-cond}, then we have
\[
\begin{split}
&\sum_{k=1}^n 2 k \mu |R(x-2 k \mu, y-2 k \nu) |^2 \\
&=2 n \mu \sum_{k=1}^n  |R(x-2 k \mu, y-2 k \nu) |^2-2 \mu \sum_{k=1}^{n-1} \sum_{j=1}^k  |R(x-2 j \mu, y-2 j \nu) |^2 \\
&=2 n \mu |R(x-2 \mu)|^2-2 \mu \{(n-1) |R(x-2\mu)|^2\}\\
&=2 \mu |R(x-2 \mu)|^2 \\
&=2 \mu (1-|R(x)|^2).
\end{split}
\]
Similarly, we have
\[\begin{split}
&\sum_{k=-n}^{-1} 2 k \mu  |R(x-2 k \mu, y-2 k \nu) |^2=-\sum_{k=1}^n 2 k \mu  |R(x+2 k \mu, y+2  k \nu) |^2 \\
&=-2 n \mu \sum_{k=1}^n  |R(x+2 k \mu, y+2  k \nu) |^2 + 2 \mu \sum_{k=1}^{n-1} \sum_{j=1}^k  |R(x+2 k \mu, y+2  k \nu) |^2\\
&=-2 n \mu |R(x+2 \mu)|^2+2 (n-1) |R(x+2 \mu)|^2\\
&=-2 \mu |R(x+2\mu)|^2\\
&=0.
\end{split}\]
%}

%{\color{blue} 
By combining the above two equations, we obtain 
\[\begin{split}
 & \sum_{k=-n}^n 2 k \mu |R(x-2 k \mu, y-2 k \nu) |^2 \\
 & =  \sum_{k=1}^n 2 k \mu |R(x-2 k \mu, y-2 k \nu) |^2 + \sum_{k=-n}^{-1} 2 k \mu |R(x-2 k \mu, y-2 k \nu) |^2  \\
 &=2 \mu (1-|R(x)|^2).
\end{split}\]

Thus the second integral of the trace in \eqref{eq:trace_int} gives
\[\begin{split}
&\frac{c}{2\mu} \int_0^{2\mu} \int_0^1\sum_{k\in \mathbb{Z}} 2 k \mu |R(x-2 k \mu, y-2 k \nu) |^2 \, dx\, dy \\
&=\frac{c}{2\mu} \int_0^{2\mu} \int_0^1 \lim_{n\to \infty} \sum_{k=-n}^n (2 k \mu) |R(x-2 k \mu, y-2 k \nu) |^2 \, dx\, dy\\
&= \frac{c}{2\mu} \int_0^{2\mu} \int_0^1 2 \mu (1-|R(x)|^2)\; dx\, dy \\
&=\frac{c}{2\mu} 2\mu 2 \mu \int_0^1 (1-|R(x)|^2) \,dx\\
&=2 c \mu  \int_0^1 (1-|R(x)|^2) \, dx.
\end{split}\]

If we let $M= 2 c \mu  \int_0^1 (1-|R(x)|^2) \,dx$, then we have
\begin{equation}\label{eq:S(Q)-nabla0}
\begin{split}
S(Q) & \ge  | 1 -\frac{c}{2\pi i} \tau_E (\langle R,   \nabla_Z (R)\rangle_L^E |=|1-\frac{c}{2} +M |.
 \end{split}
\end{equation}
%}

\end{example}

\begin{example}
Here we compute the lower bound of $S(Q)$ explicitly with another set of Yang-Mills connections $\nabla^1=\nabla^0+\mathbb{H}$, where $\nabla^0$ is the connection in \eqref{eq:Lee_conn} and $\mathbb{H}$ is the linear map from $\mathfrak{h}$ to the set of skew-symmetric elements of $\EE$ given in \eqref{eq:skew-H}. Note that $\mathbb{H}$ is given by
\begin{equation}
\begin{split}
&\mathbb{H}_X (x,y,p)= i\, g_1(y) \delta_0(p)\\
&\mathbb{H}_Y(x,y,p)= i\, g_2(x) \delta_0(p)\\
&\mathbb{H}_Z(x,y,p)=0,
\end{split}
\end{equation}
where $(x,y)\in \RR\times \TT$, and $g_1$, $g_2$ are real-valued differentiable functions satisfying $g_1(y)=g_1(y-2p\nu)$, $g_2(x)=g_2(x-2p\mu)$. There are many such functions $g_1$ and $g_2$.

As in \eqref{eq:curv-nabla-1}, the corresponding curvature of $\nabla^1$ is given by 
\[
\Theta_{\nabla^1}(X,Y)=0,\quad \Theta_{\nabla^1}(X,Z)=0,\quad \Theta_{\nabla^1}(Y,Z)=\frac{\pi i}{\mu}\Id_E.
\]
Also note that $\nabla_Z^1=\nabla_Z^0 + \mathbb{H}_Z$. Since $\mathbb{H}_Z=0$, $\nabla_Z^1=\nabla_Z^0$. Hence  the lower bound of the functional $S(Q)$ will be the same the that of $\nabla^0$ as in \eqref{eq:S(Q)-nabla0} of Example~\ref{Ex:1} above:
\begin{equation}\label{eq:S(Q)-nabla1}
S(Q) \ge |1-\frac{c}{2}  |.
\end{equation}

\end{example}

\begin{example}
Here we compute the lower bound of $S(Q)$ explicitly with the connection $\nabla^2$ given in \eqref{eq:nabla-2} that has constant curvature but neither attains minimum nor gives a critical point. As in \eqref{eq:curv-nabla-2}, the curvature is given by
\[
\Theta_{\nabla^2}(X,Y)=\nu i \Id_E,\quad \Theta_{\nabla^2}(X,Z)=0,\quad \Theta_{\nabla^2}(Y,Z)=\frac{\pi i}{\mu}\Id_E.
\]
Then Proposition~\ref{prop:1} given 
\[\begin{split}
c_\nabla(Q) &=-\frac{1}{2\pi i} \, \tau_E(\langle R, (\nu i \Id_E+c\nabla^2_Z + \frac{\pi i}{\mu} \Id_E)\cdot R\rangle_L^E)\\
&= -\frac{c}{2\pi i} \tau_E (\langle R,   \nabla^2_Z (R)\rangle_L^E +1 + \frac{\nu}{2\pi}.
\end{split}\]
Thus by Lemma~\ref{lem:1}, we have
\[
S(Q)\ge |c_\nabla(Q)| =  | -\frac{c}{2\pi i} \tau_E (\langle R,   \nabla^2_Z (R)\rangle_L^E) +1 + \frac{\nu}{2\pi}|.
\]
Since $\nabla_Z^2=\nabla_Z^0$, we have $\tau_E (\langle R,   \nabla^2_Z (R)\rangle_L^E) = \tau_E (\langle R,   \nabla^0_Z (R)\rangle_L^E)$, and hence
\begin{equation}\label{eq:S(Q)-nabla2}
S(Q) \ge |1 + \frac{\nu}{2\pi} -\frac{c}{2}  |.
\end{equation}

\end{example}

\begin{example}
Here we compute the lower bound of $S(Q)$ explicitly with the connection $\nabla^3$ with non-constant curvature given by $\nabla^3=\nabla^0+\mathbb{G}$, where $\mathbb{G}$ is given in \eqref{eq:G}. As shown in \eqref{eq:conn-nabla3}, the corresponding curvature is given by
\[
\Theta_{\nabla^3}(X,Y)=-c \mathbb{G}_Z,\quad \Theta_{\nabla^3}(X,Z)=0, \quad \Theta_{\nabla^3}(Y,Z)=\frac{\pi i}{\mu} \Id_E + \frac{\partial \mathbb{G}_Z}{\partial x}.
\]
Using Proposition~\ref{prop:1} and the fact that $\nabla^3_Z=\nabla^0_Z$, we have
\begin{equation}\label{eq:c-1-nabla3}\begin{split}
c_\nabla(Q)&=-\frac{1}{2\pi i} \tau_E(\langle R, (-c\mathbb{G}_Z+c \nabla^3_Z + \frac{\pi i}{\mu} \Id_E + \frac{\partial \mathbb{G}_Z}{\partial x} ) \cdot R \rangle_L^E)\\
&=1- \frac{c}{2\pi i} \tau_E(\langle R, \nabla_Z^0\rangle_L^E) + \frac{c}{2\pi i} \tau_E ( \langle R, \mathbb{G}_Z\cdot R\rangle_L^E) - \frac{1}{2 \pi i}\tau_E(\langle R, \frac{\partial \mathbb{G}_Z}{\partial x} \cdot R\rangle_L^E).
\end{split}\end{equation}
We compute the last two terms of the above equation as follows.
\[
A:=\frac{c}{2\pi i} \tau_E ( \langle R, \mathbb{G}_Z\cdot R\rangle_L^E)=\frac{c}{2\pi i} \int_0^{2\mu} \int_0^1 \langle R, \mathbb{G}_Z \rangle_L^E (x,y,0) \, dx\, dy.
\]
Note that 
\[
\langle R, \mathbb{G}_Z \rangle_L^E (x,y,0)=\sum_{k\in \mathbb{Z}} \overline{R}(x-2 k \mu, y-2 k \nu)(\mathbb{G}_Z\cdot R)(x-2 k \mu,  y-2 k \nu).
\]
Since $\mathbb{G}_Z(x,y,p)= i \cos(\frac{\alpha \pi x}{\mu}) \delta_0(p)$, we have
\[
(\mathbb{G}_Z\cdot R) (x-2 k \mu,  y-2 k \nu)=-i \cos \Big(\frac{\alpha \pi (x-2 k \mu)}{\mu}\Big) R(x-2 k \mu,  y-2 k \nu).
\]
Thus we have
\[
\langle R, \mathbb{G}_Z \rangle_L^E (x,y,0)=-i \sum_{k\in \mathbb{Z}} \cos \Big(\frac{\alpha \pi (x-2 k \mu)}{\mu}\Big) |R(x-2 k \mu,  y-2 k \nu)|^2,
\]
and hence
\[
A:=\frac{c}{2\pi i} \tau_E ( \langle R, \mathbb{G}_Z\cdot R\rangle_L^E) = -\frac{c}{2\pi} \int_0^{2\mu} \int_0^1 \sum_{k\in \mathbb{Z}} \cos \Big(\frac{\alpha \pi (x-2 k \mu)}{\mu}\Big) |R(x-2 k \mu,  y-2 k \nu)|^2 \, dx \, dy.
\]
To compute the last term of \eqref{eq:c-1-nabla3}, we let
\[
B:= - \frac{1}{2 \pi i}\tau_E(\langle R, \frac{\partial \mathbb{G}_Z}{\partial x} \cdot R\rangle_L^E)=-\frac{1}{2\pi i} \int_0^{2\mu} \int_0^1 \langle R, \frac{\partial \mathbb{G}_Z}{\partial x} \cdot R\rangle_L^E (x,y,0)\, dx\, dy.
\]
Note that $\frac{\partial \mathbb{G}_Z}{\partial x} (x,y,p)=-i \sin(\frac{\alpha \pi x}{\mu} \frac{\alpha \pi}{\mu} \delta_0(p)$. So we have
\[\begin{split}
\big\langle R, & \frac{\partial \mathbb{G}_Z}{\partial x} \cdot R\big\rangle_L^E (x,y,0)=\sum_{k\in \mathbb{Z}} \overline{R}(x-2 k \mu, y-2  k \nu) (\frac{\partial \mathbb{G}_Z}{\partial x} \cdot R)(x-2 k \mu, y-2  k \nu)\\
&=\sum_{k\in \mathbb{Z}} \overline{R}(x-2 k \mu, y-2  k \nu) i \sin\Big(\frac{\alpha \pi(x-2 k \mu)}{\mu}\Big) \frac{\alpha \pi}{\mu} R(x-2 k \mu, y-2  k \nu)\\
&=\frac{i \alpha \pi}{\mu} \sum_{k\in \mathbb{Z}} \sin\Big(\frac{\alpha \pi(x-2 k \mu)}{\mu}\Big) |R(x-2 k \mu, y-2  k \nu)|^2.
\end{split}\]

Thus 
\[\begin{split}
B&:= - \frac{1}{2 \pi i}\tau_E\big(\big\langle R, \frac{\partial \mathbb{G}_Z}{\partial x} \cdot R\big\rangle_L^E\big)\\
&=-\frac{1}{2 \pi i} \int_0^{2\mu} \int_0^1 \frac{i \alpha \pi}{\mu} \sum_{k\in \mathbb{Z}} \sin\Big(\frac{\alpha \pi(x-2 k \mu)}{\mu}\Big) |R(x-2 k \mu, y-2  k \nu)|^2 \, dx\, dy \\
&= -\frac{\alpha}{2\mu}  \int_0^{2\mu} \int_0^1 \sum_{k\in \mathbb{Z}} \sin\Big(\frac{\alpha \pi(x-2 k \mu)}{\mu}\Big) |R(x-2 k \mu, y-2  k \nu)|^2 \, dx\, dy. 
\end{split}\]

Therefore, the lower bound of $S(Q)$ with the non-constant curvature connection $\nabla^3$ is given by
\[\begin{split}
S(Q) &\ge |c_\nabla(Q)| =| 1- \frac{c}{2\pi i} \tau_E(\langle R, \nabla_Z^0\rangle_L^E) + \frac{c}{2\pi i} \tau_E ( \langle R, \mathbb{G}_Z\cdot R\rangle_L^E) - \frac{1}{2 \pi i}\tau_E(\langle R, \frac{\partial \mathbb{G}_Z}{\partial x} \cdot R\rangle_L^E) | \\
&= | 1- \frac{c}{2\pi i} \tau_E(\langle R, \nabla_Z^3\rangle_L^E) +A+B| \\
&= \Big|1- \frac{c}{2\mu} \int_0^{2\mu} \int_0^1\sum_{k\in \mathbb{Z}} (x-2 k \mu) |R(x-2 k \mu, y-2 k \nu) |^2 \, dx\, dy \\
& \quad -\frac{c}{2\pi} \int_0^{2\mu} \int_0^1 \sum_{k\in \mathbb{Z}} \cos \Big(\frac{\alpha \pi (x-2 k \mu)}{\mu}\Big) |R(x-2 k \mu,  y-2 k \nu)|^2 \, dx \, dy \\
& \quad -\frac{\alpha}{2\mu}  \int_0^{2\mu} \int_0^1 \sum_{k\in \mathbb{Z}} \sin \Big(\frac{\alpha \pi(x-2 k \mu)}{\mu}\Big) |R(x-2 k \mu, y-2  k \nu)|^2 \, dx\, dy \Big|.
\end{split}\]

%{\color{blue} 
Now we compute the last two integrals with the special function $R$ and particular values of $\mu$, $\alpha$ and $c$.
First choose a positive integer $N$ such that $\mu=\displaystyle{\frac{\pi}{ N}}<\frac{1}{2}$. %Note that $\hslash$ is the Planck constant.
Then choose $c=N$ and $\alpha=\displaystyle{\frac{c\mu}{\pi}}$. We compute
\[\begin{split}
&\sum_{k\in \mathbb{Z}} \cos \Big(\frac{\alpha \pi (x-2 k \mu)}{\mu}\Big) |R(x-2 k \mu,  y-2 k \nu)|^2 + \sum_{k\in \mathbb{Z}} \sin\Big(\frac{\alpha \pi(x-2 k \mu)}{\mu}\Big) |R(x-2 k \mu, y-2  k \nu)|^2\\
&=\cos \Big(\frac{\alpha \pi x}{\mu}\Big) |R(x)|^2 + \cos \Big(\frac{\alpha \pi x -\alpha \pi 2  \mu}{\mu}\Big) |R(x-2 \mu)|^2 + \sin\Big(\frac{\alpha \pi x}{\mu}\Big) |R(x)|^2 \\
&\quad + \sin \Big(\frac{\alpha \pi x -\alpha \pi 2  \mu}{\mu}\Big) |R(x-2 \mu)|^2\\
&=\cos (Nx) |R(x)|^2 + \cos (Nx-2\pi)|R(x-2 \mu)|^2 + \sin (Nx) |R(x)|^2 +\sin (Nx-2\pi) |R(x-2 \mu)|^2\\
&=(\cos (Nx) + \sin (Nx))|R(x)|^2 + (\cos (Nx-2\pi) + \sin (Nx-2\pi))|R(x-2 \mu)|^2\\
&=\sqrt{2} \sin \Big(Nx+\frac{\pi}{4}\Big)|R(x)|^2 +\sqrt{2} \sin \Big(Nx-2\pi +\frac{\pi}{4}\Big) (1-|R(x)|^2)\\
&=\sqrt{2} \sin \Big(Nx+\frac{\pi}{4}\Big)|R(x)|^2 + \sqrt{2} \sin\Big(Nx+\frac{\pi}{4}\Big)(1-|R(x)|^2)\\
&=\sqrt{2} \sin \Big(Nx+\frac{\pi}{4}\Big).
\end{split}\]

Thus we can simplify the last two integrals as follows.
\[\begin{split}
&-\frac{c}{2\pi} \int_0^{2\mu} \int_0^1 \sum_{k\in \mathbb{Z}} \cos \Big(\frac{\alpha \pi (x-2 k \mu)}{\mu}\Big) |R(x-2 k \mu,  y-2 k \nu)|^2 \, dx \, dy \\
& \quad -\frac{\alpha}{2\mu}  \int_0^{2\mu} \int_0^1 \sum_{k\in \mathbb{Z}} \sin\Big(\frac{\alpha \pi(x-2 k \mu)}{\mu}\Big) |R(x-2 k \mu, y-2  k \nu)|^2 \, dx\, dy \\
&=-\frac{N}{2\pi} \int_0^{\frac{2\pi}{ N}} \int_0^1 \sqrt{2} \sin \Big(Nx+\frac{\pi}{4}\Big)\, dx\, dy\\
&= (1+\sin N - \cos N).
\end{split}\]

%Therefore
%\[
%S(Q)\ge |1-\frac{N}{2} - \frac{1}{\hslash} (1+\sin N - \cos N)|
%\]

%}

\end{example}

\end{document}